\documentclass[a4paper]{amsart}                % American Mathematical Society document class for articles
\usepackage[latin1]{inputenc}                  % Codificación europea del teclado (acentos, símbolos teclado ...)
\usepackage[T1]{fontenc}                       % Acentos de palabras y división silábica
\usepackage[spanish,english]{babel}            % Idioma principal el ingles
\usepackage{amsmath,amssymb,amsthm}            % Paquetes de la American Mathematical Society
\usepackage{latexsym}                          % Símbolos
\usepackage{delarray}                          % Delimiter array: permite integrar delimitadores en las opciones del entorno array
\usepackage{bbm}                               % Fuentes para los números reales, racionales, ...
\usepackage{hyperref}                          % Enlaces de "hipertexto"
\usepackage[pdftex,usenames,dvipsnames]{color} % Muy importante para el color
\usepackage{bbding,trfsigns}                   % Símbolos
\usepackage{wasysym}                           % Variar el tamaño de los símbolos
\usepackage{datetime}                          % Imprimir fecha y hora

\DeclareMathAlphabet{\mathpzc}{OT1}{pzc}{m}{it} % Tipo de letra \mathpzc

\newtheorem{Th}{Theorem}[section]              % Enumera los teoremas de acuerdo con la sección (Theorem 1.1, Theorem 1.2 , ...)
\newtheorem*{Th2}{Theorem}                     % Teoremas sin numerar

\newtheorem{Prop}[Th]{Proposition}
\newtheorem{Lem}{Lemma}[section]

\newcommand{\B}{\mathbb{B}}
\newcommand{\E}{\mathbb{E}}
\newcommand{\F}{\mathbb{F}}
\newcommand{\G}{\Gamma}
\DeclareMathOperator{\supp}{supp}
\DeclareMathOperator{\spann}{span}

\title[Uniformly convex and smooth spaces and Carleson measures in Bessel settings]
      {Characterization of uniformly convex and smooth Banach spaces by using Carleson measures in Bessel settings}

\author[J.J. Betancor]{Jorge J. Betancor}
\author[A.J. Castro]{Alejandro J. Castro}
\author[L. Rodríguez-Mesa]{Lourdes Rodríguez-Mesa}

\address{\newline
        Jorge J. Betancor, Alejandro J. Castro, and Lourdes Rodríguez-Mesa \newline
        Departamento de Análisis Matemático,
        Universidad de la Laguna, \newline
        Campus de Anchieta, Avda. Astrofísico Francisco Sánchez, s/n, \newline
        38271, La Laguna (Sta. Cruz de Tenerife), Spain}

\email{jbetanco@ull.es, ajcastro@ull.es, lrguez@ull.es}



\thanks{The authors are partially supported by MTM2010/17974. The second author is also supported by a FPU grant from the Government of Spain}

\begin{document}

%  \footnotetext{Date: \today.}

  \maketitle                % Si no se activa esta opción no se pone ni el título ni los autores en el encabezado de cada página

  \begin{abstract}
    In this paper we obtain new characterizations of the $q$-uniformly convex and smooth Banach spaces by using Carleson measures.
    These measures are defined by Poisson integral associated with Bessel operators and Banach valued $BMO$-functions. By the way we describe
    $q$-uniformly convexity and smoothness of a Banach space in terms of the mapping properties of the Lusin integral defined by the Poisson
    semigroup for Bessel operators.
  \end{abstract}

    %%%%%%%%%%%%%%%%%%%%%%%%%%%%%%%%%%%%%%%%%%%%%%%%%%%%%%%%%%%%%%%%%%%%%%%%%%%%%%%%%%%%%%%%%%%%%%%%%%%%%%%%%%%%%%%%%%%%
    \section{Introduction} \label{sec:intro}
    %%%%%%%%%%%%%%%%%%%%%%%%%%%%%%%%%%%%%%%%%%%%%%%%%%%%%%%%%%%%%%%%%%%%%%%%%%%%%%%%%%%%%%%%%%%%%%%%%%%%%%%%%%%%%%%%%%%%

    It is well-known that vector valued harmonic analysis and geometry of Banach spaces are closely connected. Some geometric properties of a Banach
    space $\B$ are characterized by the boundedness in $\B$-valued $L^p$ or $BMO$ spaces of some harmonic analysis operators
    (Riesz transforms, imaginary powers, Littlewood-Paley $g$-functions, \dots). These properties have also a description by using
    martingales transforms. The celebrated papers of Bourgain \cite{Bou} and Burkholder \cite{Bu} concerning to $UMD$ (\emph{Unconditional Martingale Difference}) spaces contain the first main
    results of this theory. In the last years this area has a great activity. In \cite{Xu} Xu studied the one side Littlewood-Paley theory for
    Banach valued functions and he obtained new characterizations for the uniformly convex and smooth Banach spaces. The results in \cite{Xu}
    were generalized by Martínez, Torrea and Xu \cite{MTX} to the diffusion semigroup setting. Harmonic analysis operators associated with
    Bessel, Hermite, Laguerre and Ornstein-Uhlenbeck  operators allow also to characterize $UMD$, convexity and smoothness properties of
    Banach spaces (see \cite{AMST}, \cite{AST}, \cite{BFMT}, \cite{BFR}, \cite{BFRST1}, \cite{HTV}, amongst others).\\

    Recently, Ouyang and Xu \cite{OX} have studied the relationship between vector valued $BMO$ functions and the Carleson measures defined by their
    Poisson integrals. They obtained new characterizations for those Banach spaces admitting an equivalent norm which is $q$-uniformly convex or
    $q$-uniformly smooth.\\

    In this paper we use the Poisson integrals associated with Bessel operators to define Carleson measures that allow us to characterize
    (modulus renorming) $q$-uniformly convex and smooth Banach space. We consider the Banach valued odd $BMO$ functions on $\mathbb{R}$.
    In \cite{BCFR2} the scalar space of odd $BMO$ functions was described by using Carleson measures.\\

    Assume that $\B$ is a Banach space. We say that a locally integrable function $f: \mathbb{R} \longrightarrow \B$ has bounded mean oscillation,
    written $f \in BMO(\mathbb{R},\B)$, when
    $$\|f\|_{BMO(\mathbb{R},\B)} = \sup_{I \subset \mathbb{R}} \frac{1}{|I|} \int_I \|f(x)-f_I\|_\B dx,$$
    where the supremum is taken over all bounded intervals $I$ in $\mathbb{R}$. Here $f_I = \frac{1}{|I|} \int_I f(x)dx$, where the integral is understood in the Bochner sense, and $|I|$ denotes the length of $I$, for every bounded interval $I$ in $\mathbb{R}$. By $BMO_o(\mathbb{R},\B)$ we represent the space of the odd functions in $BMO(\mathbb{R},\B)$. According to the well-known John-Nirenberg property we can see that a $\B$-valued locally integrable and odd function $f$ on $\mathbb{R}$ is in $BMO_o(\mathbb{R},\B)$ if, and only if, for some (equivalently, for any)
    $1 \leq p < \infty$, there exists $C>0$ such that
    \begin{equation} \label{1.1}
        \left( \frac{1}{|I|} \int_I \|f(x)-f_I\|_\B^p dx \right)^{1/p}
            \leq C,
    \end{equation}
    for every interval $I=(a,b)$, $0<a<b<\infty$, and
    \begin{equation}\label{1.2}
        \left( \frac{1}{|I|} \int_I \|f(x)\|_\B^p dx \right)^{1/p}
            \leq C,
    \end{equation}
    for each interval $I=(0,b)$, $0<b<\infty$. Moreover, for every $f \in BMO_o(\mathbb{R},\B)$ and $1\leq p < \infty$,
    $\|f\|_{BMO(\mathbb{R},\B)} \simeq \inf \{ C>0, \text{\eqref{1.1} and \eqref{1.2} hold}\}$.
    As in the classical case if $f \in BMO(\mathbb{R},\B)$, then $\int_0^\infty \|f(x)\|_\B / (1+x^2)dx<\infty$.\\

    In \cite{MS} Muckenhoupt and Stein developed the harmonic analysis theory in the ultraspherical and Bessel settings. Taking as a starting point
    the ideas in \cite{MS} in the last years several authors have investigated boundedness properties of harmonic analysis operators associated
    with Bessel operators  (\cite{BBFMT}, \cite{BCC1}, \cite{BCN}, \cite{BFS}, \cite{BS}, \cite{NS}).\\

    We consider, for every $\lambda>0$, the Bessel operator $\Delta_\lambda = -x^{-\lambda} \frac{d}{dx}x^{2\lambda}\frac{d}{dx}x^{-\lambda}$,
    $x \in (0,\infty)$, and the Hankel transformation $h_\lambda$ defined by
    $$h_\lambda(f)(x)=\int_0^\infty \sqrt{xy} J_{\lambda-1/2}(xy) f(y) dy, \quad x \in (0,\infty),$$
    for every $f \in L^1(0,\infty) \cap L^2(0,\infty)$. Here, $J_\nu$ denotes the Bessel function of the first kind and order $\nu$. $h_\lambda$ can be extended to $L^2(0,\infty)$ as an isometry of $L^2(0,\infty)$ where $h_\lambda^{-1}=h_\lambda$. If $f \in C_c^\infty(0,\infty)$, the space of
    smooth functions with compact support, we have that
    $$h_\lambda(\Delta_\lambda f)(y)=y^2 h_\lambda(f)(y), \quad y \in (0,\infty).$$
    We define the operator $\tilde{\Delta}_\lambda$ as follows,
    $$\tilde{\Delta}_\lambda f = h_\lambda(y^2 h_\lambda (f)), \quad f \in D(\tilde{\Delta}_\lambda),$$
    where the domain $D(\tilde{\Delta}_\lambda)$ of $\tilde{\Delta}_\lambda$ is
    $$D(\tilde{\Delta}_\lambda) = \{f \in L^2(0,\infty) : y^2 h_\lambda (f) \in L^2(0,\infty)\}.$$
    $\tilde{\Delta}_\lambda$ is a closed and positive operator. Note that $\tilde{\Delta}_\lambda f = \Delta_\lambda f$, $f \in C_c^\infty(0,\infty)$.
    In the sequel we refer to $\tilde{\Delta}_\lambda$ also by $\Delta_\lambda$.\\

    By $\{P_t^\lambda\}_{t>0}$ we represent the Poisson semigroup associated with $\Delta_\lambda$, or, in other words, the semigroup of operators generated by
    $-\sqrt{\Delta_\lambda}$. According to \cite[(16.4)]{MS} we can write, for every $f \in L^p(0,\infty)$, $1 \leq p \leq \infty$,
    $$P_t^\lambda(f)(x)=\int_0^\infty P_t^\lambda(x,y)f(y)dy, \quad x,t \in (0,\infty),$$
    where
    \begin{equation}\label{Bessel_kernel}
        P_t^\lambda(x,y)=\frac{2\lambda (xy)^\lambda t}{\pi} \int_0^\pi \frac{(\sin \theta)^{2\lambda-1}}{[(x-y)^2+t^2+2xy(1-\cos \theta)]^{\lambda+1}} d\theta, \quad x,y,t \in (0,\infty).
    \end{equation}
    $\{P_t^\lambda\}_{t>0}$ is a contractive semigroup in $L^p(0,\infty)$, $1 \leq p \leq \infty$. Since the kernel function $P_t^\lambda(x,y)\geq 0$,
    $x,y,t \in (0,\infty)$, the operator $P_t^\lambda$ is also a contraction in the Lebesgue-Bochner space $L^p((0,\infty),\B)$, for every $1 \leq p \leq \infty$
    and $t>0$.\\

    We say that a positive measure $\mu$ on $(0,\infty) \times (0,\infty)$ is Carleson when there exists $C>0$ satisfying
    $$\frac{\mu(I \times (0,|I|))}{|I|} \leq C,$$
    for every bounded interval $I$ in $(0,\infty)$. It is well-known that the functions of bounded mean oscillation on $\mathbb{R}^n$ can be characterized by using
    Carleson measures. In \cite{BCFR2} the following result was established.

    \begin{Th2}[{\cite[Theorem 1.1]{BCFR2}}]
        Let $\lambda>0$. Assume that $f$ is a locally integrable function in $[0,\infty)$. If we define $f_o$ as the odd extension of $f$ to $\mathbb{R}$, then
        $f_o \in BMO_o(\mathbb{R})$ if and only if $(1+x^2)^{-1}f \in L^1(0,\infty)$ and the measure $\gamma_f$ given by
        $$d\gamma_f(x,t) = \left| t \partial_t P_t^\lambda(f)(x) \right|^2 \frac{dxdt}{t}$$
        is Carleson on $(0,\infty) \times (0,\infty)$.
    \end{Th2}

    We now recall the definitions of convexity and smoothness for a Banach space $\B$. The modulus of convexity $\delta_\B$ and of smoothness $\rho_\B$ are defined by
    $$\delta_\B(\varepsilon)=\inf\{1-\| \frac{a+b}{2} \|_\B : a,b \in \B, \|a\|_\B=\|b\|_\B=1, \|a-b\|_\B=\varepsilon \}, \quad 0< \varepsilon < 2,$$
    and
    $$\rho_\B(t)=\sup\{ \frac{\|a+tb\|_\B + \|a-tb\|_\B}{2} : a,b \in \B, \|a\|_\B=\|b\|_\B=1 \}, \quad t>0.$$
    We say that $\B$ is uniformly convex (respectively, uniformly smooth) when $\delta_\B(\varepsilon)>0$ (respectively, $\lim_{t \to 0} \rho_\B(t)/t=0$).
    Also, $\B$ is called $q$-uniformly convex, $q \geq 2$ (respectively, $q$-uniformly smooth, $1<q\leq 2$) when there exist $C>0$ such that $\delta_\B(\varepsilon) \geq C \varepsilon^q$,
    $0<\varepsilon<2$ (respectively, $\rho_\B(t) \leq C t^q$, $t>0$).\\

    Pisier \cite{Pi} proved that $\B$ has an equivalent norm that is $q$-uniformly convex (respectively, $q$-uniformly smooth) if and only if $\B$ has martingale
    cotype $q$ (respectively, martingale type $q$). Xu \cite{Xu} established the corresponding characterization when the martingale type and cotype is replaced by the
    Lusin type and cotype associated with the Poisson semigroup for the torus. The result of Xu was extended to the diffusion semigroup setting in \cite{MTX}.
    Similar properties have been obtained in the Bessel (\cite{BFMT}) and Laguerre (\cite{BFRST1}) contexts.
    Recently, Ouyang and Xu \cite{OX} characterized those Banach spaces having an equivalent norm that is $q$-uniformly convex or smooth by using Carleson measures and $\B$-valued $BMO$ functions, and lately Jiao \cite{Ji} gave the martingale version of this result. \\

    In this paper we obtain new characterizations for $q$-uniformly convexity and smoothness of a Banach space by using Carleson measures associated with the Bessel
    Poisson integrals $P_t^\lambda(f)$ of $f$ belonging to $BMO_o(\mathbb{R},\B)$.\\

    The main results of this paper are the following ones.

    \begin{Th}\label{Th_3.1_OX}
        Let $\B$ be a Banach space, $\lambda>1$ and $2 \leq q < \infty$. Then, the following statements are equivalent.
        \begin{itemize}
            \item[$(i)$] There exists $C>0$ such that, for every $f \in BMO_o(\mathbb{R},\B)$, the measure $d\mu_f$ defined by
            $$d\mu_f(x,t)=\|t \partial_t P_t^\lambda (f)(x)\|_\B^q \frac{dxdt}{t}$$
            is Carleson on $(0,\infty)\times (0,\infty)$ and
            $$  \sup_{I} \frac{1}{|I|} \int_0^{|I|} \int_I \|t \partial_t P_t^\lambda(f)(x) \|^q_\B \frac{dx dt}{t}
                    \leq C \|f\|^q_{BMO_o(\mathbb{R},\B)},$$
            where the supremum is taken over all bounded intervals $I$ in $(0,\infty)$. \\ \quad
            \item[$(ii)$] $\B$ has an equivalent norm which is $q$-uniformly convex.
        \end{itemize}
    \end{Th}

    \begin{Th}\label{Th_4.1_OX}
        Let $\B$ be a Banach space, $\lambda>1$ and $1 < q \leq 2$. Then, the following assertions are equivalent.
        \begin{itemize}
            \item[$(i)$] There exists $C>0$ such that, for every odd
             $\B$-valued function $f$, satisfying that $(1+x^2)^{-1}f \in L^1(\mathbb{R},\B)$,
            $$  \|f\|^q_{BMO_o(\mathbb{R},\B)}
                    \leq C  \sup_{I} \frac{1}{|I|}\int_0^{|I|} \int_I \|t \partial_t P_t^\lambda(f)(x) \|^q_\B \frac{dx dt}{t},$$
            where the supremum is taken over all bounded intervals $I$ in $(0,\infty)$. \\ \quad
            \item[$(ii)$] $\B$ has an equivalent $q$-uniformly smooth norm.
        \end{itemize}
    \end{Th}

    Theorems~\ref{Th_3.1_OX} and \ref{Th_4.1_OX} can be seen as versions of \cite[Theorems 3.1 and 4.1]{OX}, respectively.\\

    In order to prove our theorems we need to obtain characterizations of the convexity and smoothness for a Banach space by using certain area
    integrals involving the Poisson semigroup $\{P_t^\lambda\}_{t>0}$.\\

    We define the following sets
    $$\G(x)=\{(y,t) \in \mathbb{R}\times (0,\infty) : |y-x|<t\}, \quad x \in \mathbb{R},$$
    and
    $$\G_+(x)=\{(y,t) \in (0,\infty)\times (0,\infty) : |y-x|<t\}, \quad x \in (0,\infty).$$
    We extend the definition of the Poisson kernel $P_t^\lambda(x,y)$ given in \eqref{Bessel_kernel} to $\mathbb{R} \times \mathbb{R}$,
    for every $t>0$, as follows
    \begin{equation*}
        P_t^\lambda(x,y)=\frac{2\lambda |xy|^\lambda t}{\pi} \int_0^\pi \frac{(\sin \theta)^{2\lambda-1}}{[(x-y)^2+t^2+2xy(1-\cos \theta)]^{\lambda+1}} d\theta.
    \end{equation*}
    Note that, for every $t>0$,
    \begin{equation}\label{(*)}
        P_t^\lambda(x,y)=P_t^\lambda(-x,y)=P_t^\lambda(x,-y), \quad x,y \in \mathbb{R}.
    \end{equation}
    We consider the Lusin integrals associated with the Poisson semigroup $\{P_t^\lambda\}_{t>0}$ defined by
    \begin{equation*}
        S_{\lambda}^q(f)(x)= \left( \int_{\G(x)} \left\| t \partial_t P_t^\lambda(f)(y) \right\|_\B^q \frac{dt dy}{t^2} \right)^{1/q}, \quad x \in \mathbb{R},
    \end{equation*}
    and
    \begin{equation*}
        S_{\lambda,+}^q(f)(x)= \left( \int_{\G_+(x)} \left\| t \partial_t P_t^\lambda(f)(y) \right\|_\B^q \frac{dt dy}{t^2} \right)^{1/q}, \quad x \in (0,\infty),
    \end{equation*}
    where $q>1$ and $f$ is a strongly $\B$-valued measurable function defined on $(0,\infty)$ such that
    $$\int_0^\infty \frac{\|f(y)\|_{\B}}{1+y^2}dy<\infty.$$
    It is not hard to see that
    \begin{equation}\label{equivalentes}
        S^q_{\lambda,+}(f)(x) \leq S^q_{\lambda}(f)(x) \leq 2^{1/q} S^q_{\lambda,+}(f)(x), \quad x \in (0,\infty).
    \end{equation}

    We denote by $H^1(\mathbb{R},\B)$ the $\B$-valued Hardy space. $H^1_o(\mathbb{R},\B)$ is the subspace of $H^1(\mathbb{R},\B)$ constituted by all
    those odd functions in $H^1(\mathbb{R},\B)$. A strongly measurable $\B$-valued function $a$ defined on $\mathbb{R}$ such that
    $\int_\mathbb{R} a(x)dx=0$ is called an $\infty$-atom (respectively, a $2$-atom) when there exists a bounded interval $I$ in $\mathbb{R}$
    such that $\supp(a)\subset I$ and $\|a\|_{L^\infty(\mathbb{R},\B)} \leq 1/|I|$ (respectively, $\|a\|_{L^2(\mathbb{R},\B)} \leq 1/|I|^{1/2}$).
    According to well-known atomic representations of the elements of $H^1(\mathbb{R},\B)$ (\cite{Hy}, \cite[p. 34, 40]{J}) we can see that a strongly measurable $\B$-valued
    odd function $f$ defined on $\mathbb{R}$ is in $H_o^1(\mathbb{R},\B)$ if, and only if, $f=\sum_{j=1}^\infty \lambda_j a_j$ in $L^1((0,\infty),\B)$, where $\{\lambda_j\}_{j=1}^\infty \subset \mathbb{C}$ satisfies that $\sum_{j=1}^\infty |\lambda_j|<\infty$ and $\{a_j\}_{j=1}^\infty$
    is a sequence of strongly measurable $\B$-valued functions defined on $(0,\infty)$ such that, for every $j \in \mathbb{N}$, $a_j$ is an
    $\infty$-atom supported on $(0,\infty)$ or there exists $\beta>0$ for which $\supp(a_j) \subset [0,\beta]$ and
    $\|a\|_{L^\infty((0,\infty),\B)} \leq 1/\beta$. Also, we can characterize the elements of $H_o^1(\mathbb{R},\B)$ similarly by using $2$-atoms. By proceeding as in the proof of \cite[Theorem 2.1]{Fr} we can show that a strongly measurable $\B$-valued
    odd function $f$ defined on $\mathbb{R}$ is in $H_o^1(\mathbb{R},\B)$ if, and only if, $f=\sum_{j=1}^\infty \lambda_j a_j$ in $L^1((0,\infty),\B)$, where $\{\lambda_j\}_{j=1}^\infty \subset \mathbb{C}$ is such that $\sum_{j=1}^\infty |\lambda_j|<\infty$, and $\{a_j\}_{j=1}^\infty$
    is a sequence of strongly measurable $\B$-valued functions defined on $(0,\infty)$ such that, for every $j \in \mathbb{N}$, $a_j$ is an
    $\infty$-atom supported on $(0,\infty)$ or $a_j=b_j \chi_{(0,\delta_j)}/\delta_j$, for a certain $b_j \in \B$, being $\|b_j\|_\B=1$
    and $\delta_j>0$. Here, $\chi_{(0,\delta)}$ denotes the characteristic function of $(0,\delta)$, for every $\delta>0$. The topology of $H^1_o(\mathbb{R},\B)$ is defined by the norms associated in the usual way with the above atomic representations.\\

    In \cite{BFMT} the martingale type and cotype of a Banach space is characterized by using Littlewood-Paley $g$-functions associated with the Poisson
    semigroup $\{P_t^\lambda\}_{t>0}$. In the next result we establish the corresponding properties involving Lusin area integrals $S^q_{\lambda,+}$.
    This proposition has interest in itself and it is useful in the proof of Theorems~\ref{Th_3.1_OX} and \ref{Th_4.1_OX}.

    \begin{Prop}\label{Lem_principal}
        Let $\B$ be a Banach space, $\lambda>0$ and $2 \leq q < \infty$. Then, the following assertions are equivalent.
        \begin{itemize}
            \item[$(i)$] For some $1<p<\infty$, there exists $C>0$ such that
            $$\|S^q_{\lambda,+}(f)\|_{L^p(0,\infty)} \leq C \|f\|_{L^p((0,\infty),\B)}, \quad f \in L^p((0,\infty),\B).$$
            \item[$(ii)$] For every $1<p<\infty$, there exists $C>0$ such that
            $$\|S^q_{\lambda,+}(f)\|_{L^p(0,\infty)} \leq C \|f\|_{L^p((0,\infty),\B)}, \quad f \in L^p((0,\infty),\B).$$
            \item[$(iii)$] There exists $C>0$ for which
            $$ \|S^q_{\lambda,+}(f)\|_{L^1(0,\infty)} \leq C \|f\|_{H^1_o(\mathbb{R},\B)}, \quad f \in H^1_o(\mathbb{R},\B).$$
            \item[$(iv)$] $\B$ has an equivalent $q$-uniformly convex norm. \\ \quad
            \item[$(v)$] $\B^*$, the dual space of $\B$, has an equivalent $q'$-uniformly smooth norm, where $q'=q/(q-1)$.
        \end{itemize}
    \end{Prop}

    This paper is organized as follows. In Section~\ref{sec:propo} we prove Proposition~\ref{Lem_principal}. Proofs of Theorems~\ref{Th_3.1_OX}
    and \ref{Th_4.1_OX} are presented in Sections~\ref{sec:proof1} and \ref{sec:proof2}, respectively. In order to make
    Sections~\ref{sec:proof1} and \ref{sec:proof2} more legible we include in Section~\ref{sec:appendix} (Appendix) the proofs of some auxiliary results
    that we need to show Theorems~\ref{Th_3.1_OX} and \ref{Th_4.1_OX}.\\

    Throughout this paper by $C$ we always denote a positive constant that is not necessarily the same in each occurrence. The duality pairing between
    a Banach space $\B$ and its dual $\B^*$ will be represented by $\langle \cdot , \cdot \rangle_{\B \times \B^*}$ or simply $\langle \cdot , \cdot \rangle$.\\

    \textbf{Acknowledgements}. The authors would like to thank Professor José Luis Torrea our always helpful discussions with him about vector
    valued harmonic analysis.

%\newpage
    %%%%%%%%%%%%%%%%%%%%%%%%%%%%%%%%%%%%%%%%%%%%%%%%%%%%%%%%%%%%%%%%%%%%%%%%%%%%%%%%%%%%%%%%%%%%%%%%%%%%%%%%%%%%%%%%%%%%
    \section{Proof of Proposition~\ref{Lem_principal}}\label{sec:propo}
    %%%%%%%%%%%%%%%%%%%%%%%%%%%%%%%%%%%%%%%%%%%%%%%%%%%%%%%%%%%%%%%%%%%%%%%%%%%%%%%%%%%%%%%%%%%%%%%%%%%%%%%%%%%%%%%%%%%%

    In order to proof Proposition~\ref{Lem_principal} we use \cite[Lemma~4.2]{OX} where the convexity and smoothness of a Banach space $\B$
    is described in terms of the boundedness properties of the Lusin area integral associated with the classical Poisson integral.\\

    If $f$ is a strongly measurable $\B$-valued function defined in $\mathbb{R}$ such that $\int_\mathbb{R} \|f(x)\|_\B/(1+x^2)dx<\infty$,
    and $q>1$ the $q$-Lusin area integral $S^q(f)$ is defined by
    $$S^q(f)(x)= \left( \int_{\G(x)} \left\| t \partial_t P_t(f)(y) \right\|_\B^q \frac{dt dy}{t^2} \right)^{1/q}, \quad x \in \mathbb{R},$$
    where $P_t(f)$ represents the Poisson integral of $f$ given as follows
    $$P_t(f)(y)=\int_\mathbb{R} P_t(y-z)f(z)dz, \quad y \in \mathbb{R}, \ t>0.$$
    As usual, we denote the Poisson kernel by
    $$P_t(z)=\frac{1}{\pi} \frac{t}{t^2+z^2}, \quad z \in \mathbb{R}, \ t>0.$$
    We also consider the following partial Lusin integrals
    $$S^q_+(f)(x)= \left( \int_{\G_+(x)} \left\| t \partial_t P_t(f)(y) \right\|_\B^q \frac{dt dy}{t^2} \right)^{1/q}, \quad x \in (0,\infty),$$
    and
    $$ S_{+,loc}^q(f)(x)= \left( \int_{\G_+(x)} \left\| t \partial_t P_t(f\chi_{(x/2,2x)})(y) \right\|_\B^q \frac{dt dy}{t^2} \right)^{1/q}, \quad x \in (0,\infty),$$
    with $q>1$.\\

    We prove Proposition~\ref{Lem_principal} in two steps. Firstly, we establish that the $L^p$-boundedness of $S^q$ is equivalent to the
    $L^p$-boundedness of $S^{q}_{\lambda,+}$, for every $1<p<\infty$.

    \begin{Lem}\label{Lem_Lp}
        Let $\B$ be a Banach space, $\lambda>0$, $2 \leq q < \infty$, and $1<p<\infty$. Then, the following assertions are equivalent.
        \begin{itemize}
            \item[$(i)$]    $S^q$ is bounded from $L^p(\mathbb{R},\B)$ into $L^p(\mathbb{R})$.\\
            \item[$(ii)$]   $S^q_+$ is bounded from $L^p(\mathbb{R},\B)$ into $L^p(0,\infty)$.\\
            \item[$(iii)$]  $S^q_{+,loc}$ is bounded from $L^p((0,\infty),\B)$ into $L^p(0,\infty)$.\\
            \item[$(iv)$]   $S^q_{\lambda,+}$ is bounded from $L^p((0,\infty),\B)$ into $L^p(0,\infty)$.\\
        \end{itemize}
    \end{Lem}

    \begin{proof}

        $(i) \Rightarrow (ii)$. It is sufficient to note that $S_+^q(f)(x) \leq S^q(f)(x)$, $x \in (0,\infty)$, for every $f \in L^p(\mathbb{R},\B)$.\\

        $(ii) \Rightarrow (i)$. Let $f \in L^p(\mathbb{R},\B)$. We decompose $f$ as follows, $f=f_o + f_e$,
        where $f_o(x)=(f(x)-f(-x))/2$ and  $f_e(x)=(f(x)+f(-x))/2$, $x \in \mathbb{R}$. We have that
        $$P_t(f_o)(y) = \int_0^\infty [P_t(y-z)-P_t(y+z)]f_o(z)dz, \quad y \in \mathbb{R}, \ t>0,$$
        and
        $$P_t(f_e)(y) = \int_0^\infty [P_t(y-z)+P_t(y+z)]f_e(z)dz,  \quad y \in \mathbb{R}, \ t>0.$$
        For every $t>0$, the function $\|t\partial_t P_t(f_o)(y)\|_\B$ is even. Then, $S^q(f_o)$ is also an even function and
        \begin{equation*}\label{Sq-Sq+}
            S^q(f_o)(x) \leq 2^{1/q} S_+^q(f_o)(|x|), \quad x \in \mathbb{R}.
        \end{equation*}
        Hence, we get
        $$\|S^q(f_o)\|_{L^p(\mathbb{R})} \leq 2^{1/p+1/q} \|S^q_+(f_o)\|_{L^p(0,\infty)}.$$
        In a similar way we obtain
        $$\|S^q(f_e)\|_{L^p(\mathbb{R})} \leq 2^{1/p+1/q} \|S^q_+(f_e)\|_{L^p(0,\infty)}.$$
        The above inequalities allow us to show that $(ii)$ implies $(i)$.\\

        $(ii) \Leftrightarrow (iii)$. We are going to see that there exists $C>0$ such that
        \begin{equation}\label{2.1}
            \|S^q_+(f) - S^q_{+,loc}(f)\|_{L^p(0,\infty)} \leq C \|f\|_{L^p(\mathbb{R},\B)}, \quad f \in L^p(\mathbb{R},\B).
        \end{equation}
        Let $f \in L^p(\mathbb{R},\B)$. We can write
        $$P_t(f)(y)
            = \int_0^\infty P_t(y-z) f(z)dz + \int_0^\infty P_t(y+z) f(-z)dz, \quad y,t \in (0,\infty).$$
        By applying Minkowski's inequality we get
        \begin{align*}
            |S^q_+(f)(x) - S^q_{+,loc}(f)(x)|
                \leq & \left( \int_{\G_+(x)} \left\| t \partial_t \left(
                            \int_{\mathbb{R}} P_t(y-z)f(z)dz - \int_{x/2}^{2x} P_t(y-z)f(z)dz
                        \right) \right\|_\B^q \frac{dtdy}{t^2} \right)^{1/q} \\
                \leq & \mathcal{J}_1(\|f\|_\B)(x) + \mathcal{J}_2(\|\tilde{f}\|_\B)(x), \quad x \in (0,\infty),
        \end{align*}
        where $\tilde{f}(x)=f(-x)$, $x \in (0,\infty)$, and the operators $\mathcal{J}_i$, $i=1,2$, are defined by
        $$\mathcal{J}_1(g)(x) = \int_{(0,x/2) \cup (2x,\infty)} \left\| t \partial_t P_t(y-z) \right\|_{L^q\left(\G_+(x),\frac{dtdy}{t^2}\right)} g(z) dz, \quad x \in (0,\infty),$$
        and
        $$\mathcal{J}_2(g)(x) = \int_0^\infty \left\| t \partial_t P_t(y+z) \right\|_{L^q\left(\G_+(x),\frac{dtdy}{t^2}\right)} g(z) dz, \quad x \in (0,\infty).$$
        Our objective \eqref{2.1} will be established when we prove that the operators $\mathcal{J}_1$ and $\mathcal{J}_2$
        are bounded from $L^p(0,\infty)$ into itself.\\

        First, we observe that
        \begin{align} \label{bound1}
            & \left\| t \partial_t P_t(y-z) \right\|^q_{L^q\left(\G_+(x),\frac{dtdy}{t^2}\right)}
                = \int_{\G_+(x)} \left| t \partial_t P_t(y-z) \right|^q \frac{dt dy}{t^2}
                 \leq C \int_0^\infty \int_{|x-y|}^\infty \frac{t^{q-2}}{(|y-z|+t)^{2q}} dt dy \nonumber \\
            & \quad    \leq  C \int_0^\infty \int_{|x-y|}^\infty \frac{dt}{(|y-z|+t)^{q+2}} dy
                 \leq  C \left(\int_{I_{x,z}} \frac{dy}{(|y-z|+|x-y|)^{q+1}} + \int_{\mathbb{R} \setminus I_{x,z}} \frac{dy}{(|y-z|+|x-y|)^{q+1}} \right)  \nonumber\\
            & \quad     \leq  \frac{C}{|x-z|^q}, \quad x,z \in (0,\infty), \ x \neq z.
        \end{align}
        Here, $I_{x,z}$ represents the interval $(\min\{x,z\},\max\{x,z\})$. We also get
        \begin{align}\label{bound2}
            \left\| t \partial_t P_t(y+z) \right\|^q_{L^q\left(\G_+(x),\frac{dtdy}{t^2}\right)}
                & \leq C \left(\int_0^x + \int_x^\infty \right) \frac{dy}{(y+z+|x-y|)^{q+1}}
                 \leq \frac{C}{(x+z)^q}, \quad x,z \in (0,\infty).
        \end{align}
        These estimates lead to, for $i=1,2$,
        $$|\mathcal{J}_i(g)| \leq C \left(H_0(|g|) + H_\infty(|g|) \right),$$
        where $H_0$ and $H_\infty$ denote the Hardy type operators defined by
        $$H_0(g)(x)=\frac{1}{x}\int_0^x g(y) dy, \quad H_\infty(g)(x)=\int_x^\infty \frac{g(y)}{y} dy, \quad x \in (0,\infty).$$
        Since $H_0$ and $H_\infty$ are bounded from $L^p(0,\infty)$ into itself (see \cite{Mu}), we conclude that $\mathcal{J}_1$ and
        $\mathcal{J}_2$ are bounded from $L^p(0,\infty)$ into itself. Now the desired equivalence follows from \eqref{2.1}.\\

        $(iii) \Leftrightarrow (iv)$. This property will be proved when we show that there exists $C>0$ such that
        \begin{equation}\label{7.1}
            \|S^q_{+,loc}(f) - S^q_{\lambda,+}(f)\|_{L^p(0,\infty)} \leq C \|f\|_{L^p((0,\infty),\B)}, \quad f \in L^p((0,\infty),\B).
        \end{equation}
        We decompose the Bessel Poisson kernel as follows
        \begin{align}\label{7.2}
            P_t^\lambda(x,y)
                = & \frac{2\lambda (xy)^\lambda t}{\pi} \int_0^{\pi/2} \frac{(\sin \theta)^{2\lambda-1}}{[(x-y)^2+t^2+2xy(1-\cos \theta)]^{\lambda+1}}d\theta \nonumber \\
                & + \frac{2\lambda (xy)^\lambda t}{\pi} \int_{\pi/2}^\pi \frac{(\sin \theta)^{2\lambda-1}}{[(x-y)^2+t^2+2xy(1-\cos \theta)]^{\lambda+1}} d\theta \nonumber \\
                = & P_{t,1}^\lambda(x,y)+P_{t,2}^\lambda(x,y), \quad x,y,t \in (0,\infty).
        \end{align}
        By applying Minkowski's inequality we get
        \begin{equation}\label{2.3a}
            |S^q_{+,loc}(f) - S^q_{\lambda,+}(f)| \leq  \sum_{j=1}^3 \mathcal{K}_j(\|f\|_\B),
        \end{equation}
        where
        \begin{equation*}
            \mathcal{K}_1(g)(x) = \int_{(0,x/2) \cup (2x,\infty)} \left\| t \partial_t P_{t,1}^\lambda(y,z) \right\|_{L^q\left(\G_+(x),\frac{dtdy}{t^2}\right)} g(z) dz, \quad x \in (0,\infty),
        \end{equation*}
        \begin{equation*}
            \mathcal{K}_2(g)(x) = \int_0^\infty \left\| t \partial_t P_{t,2}^\lambda(y,z) \right\|_{L^q\left(\G_+(x),\frac{dtdy}{t^2}\right)} g(z) dz, \quad x \in (0,\infty),
        \end{equation*}
        and
        \begin{equation*}
            \mathcal{K}_3(g)(x) = \int_{x/2}^{2x} \left\| t \partial_t [P_{t,1}^\lambda(y,z)-P_t(y-z)] \right\|_{L^q\left(\G_+(x),\frac{dtdy}{t^2}\right)} g(z) dz, \quad x \in (0,\infty).
        \end{equation*}
        Our objective is to see that $\mathcal{K}_j$ is a bounded operator in $L^p(0,\infty)$, for $j=1,2,3$.\\

        According to \cite[p.  481--482]{BCFR2} we have that
        \begin{equation}\label{Pt1}
            |\partial_t P_{t,1}^\lambda(y,z)| \leq C \frac{z^\lambda}{(|z-y|+t)^{\lambda+2}}, \quad y,z,t \in (0,\infty).
        \end{equation}
        Also, we get
        \begin{equation}\label{Pt2}
            |\partial_t P_{t,2}^\lambda(y,z)| \leq C \frac{z^\lambda}{(z+y+t)^{\lambda+2}}, \quad y,z,t \in (0,\infty).
        \end{equation}

        By proceeding as in \eqref{bound1} and \eqref{bound2} we can see that
        \begin{equation}\label{||Pt1||}
            \left\| t \partial_t P_{t,1}^\lambda(y,z) \right\|_{L^q\left(\G_+(x),\frac{dtdy}{t^2}\right)}
                \leq C \frac{z^\lambda}{|x-z|^{\lambda+1}}, \quad x,z \in (0,\infty), \ x \neq z,
        \end{equation}
        and
        \begin{equation}\label{||Pt2||}
            \left\| t \partial_t P_{t,2}^\lambda(y,z) \right\|_{L^q\left(\G_+(x),\frac{dtdy}{t^2}\right)}
                \leq C \frac{z^\lambda}{(x+z)^{\lambda+1}}, \quad x,z \in (0,\infty).
        \end{equation}
        Then, we get that
        $$|\mathcal{K}_1(g)| + |\mathcal{K}_2(g)| \leq C (H_0(|g|)+H_\infty(|g|)).$$
        Therefore, $\mathcal{K}_1$ and $\mathcal{K}_2$ are bounded from $L^p(0,\infty)$ into itself.\\

        Next, we deal with the most involved operator $\mathcal{K}_3$. In order to do this we introduce the new kernels
        \begin{equation*}
            P_{t,1,1}^\lambda(y,z)=\frac{2\lambda (yz)^\lambda t}{\pi} \int_0^{\pi/2} \frac{\theta^{2\lambda-1}}{[(y-z)^2+t^2+2yz(1-\cos \theta)]^{\lambda+1}} d\theta, \quad y,z,t \in (0,\infty),
        \end{equation*}
        and
        \begin{equation*}
            P_{t,1,2}^\lambda(y,z)=\frac{2\lambda (yz)^\lambda t}{\pi} \int_0^{\pi/2} \frac{\theta^{2\lambda-1}}{[(y-z)^2+t^2+yz\theta^2]^{\lambda+1}} d\theta, \quad y,z,t \in (0,\infty).
        \end{equation*}
        We can write
        \begin{align*}
            |\mathcal{K}_3(g)(x)|
               \leq &   \int_{x/2}^{2x} \left\| t \partial_t [P_{t,1}^\lambda(y,z)-P_{t,1,1}^\lambda(y,z)] \right\|_{L^q\left(\G_+(x),\frac{dtdy}{t^2}\right)} |g(z)| dz \\
                    & + \int_{x/2}^{2x} \left\| t \partial_t [P_{t,1,1}^\lambda(y,z)-P_{t,1,2}^\lambda(y,z)] \right\|_{L^q\left(\G_+(x),\frac{dtdy}{t^2}\right)} |g(z)| dz \\
                    & + \int_{x/2}^{2x} \left\| t \partial_t [P_{t,1,2}^\lambda(y,z)-P_t(y-z)] \right\|_{L^q\left(\G_+(x),\frac{dtdy}{t^2}\right)} |g(z)| dz,
                    \quad x \in (0,\infty).
        \end{align*}
        By arguing like in \cite[p.  483--487]{BCFR2}, we deduce that, for each $x \in (0,\infty)$ and $x/2<z<2x$,
        \begin{itemize}
            \item $\displaystyle \left\| t \partial_t [P_{t,1}^\lambda(y,z)-P_{t,1,1}^\lambda(y,z)] \right\|_{L^q\left(\G_+(x),\frac{dtdy}{t^2}\right)}
                                    \leq   \frac{C}{z} \left( 1 + \log_+\frac{z}{|x-z|} \right),$ \\
            \item $\displaystyle \left\| t \partial_t [P_{t,1,1}^\lambda(y,z)-P_{t,1,2}^\lambda(y,z)] \right\|_{L^q\left(\G_+(x),\frac{dtdy}{t^2}\right)}
                                    \leq   \frac{C}{z} \left( 1 + \log_+\frac{z}{|x-z|} \right),$\\
            \item $\displaystyle \left\| t \partial_t [P_{t,1,2}^\lambda(y,z)-P_t(y-z)] \right\|_{L^q\left(\G_+(x),\frac{dtdy}{t^2}\right)}
                                    \leq   \frac{C}{z}.$ \\ \quad
        \end{itemize}
        There (in \cite[p.  483--487]{BCFR2}), the case $q=2$ is considered, but the same arguments are still valid for $2 \leq q < \infty$.\\

        Hence, we have that
        \begin{equation}\label{2.8}
            |\mathcal{K}_3(g)(x)|
                \leq C \int_{x/2}^{2x} \frac{1}{z} \left( 1 + \log_+\frac{z}{|x-z|} \right) |g(z)| dz, \quad x \in (0,\infty).
        \end{equation}
        We denote $\displaystyle C_0= \int_{1/2}^2 \frac{1}{u}\left( 1 + \log_+\frac{u}{|1-u|} \right) du$. By using Jensen's inequality we get
        \begin{align*}
            |\mathcal{K}_3(g)(x)|^p
                \leq & C \left( \int_{1/2}^1 \frac{1}{u}\left( 1 + \log_+\frac{u}{|1-u|} \right)  |g(xu)|du \right)^p \\
                \leq & C C_0^{p-1} \int_{1/2}^1 \frac{1}{u}\left( 1 + \log_+\frac{u}{|1-u|} \right) |g(xu)|^p du, \quad x \in (0,\infty).
        \end{align*}
        Then,
        $$\|\mathcal{K}_3(g)\|_{L^p(0,\infty)}
            \leq C \|g\|_{L^p(0,\infty)}, \quad g \in L^p(0,\infty).$$
        Putting together the above estimations we obtain \eqref{7.1}, and the proof of $(iii) \Leftrightarrow (iv)$ is finished.
    \end{proof}

    In the following lemma we establish the endpoint result $p=1$.

    \begin{Lem}\label{Lem_L1}
        Let $\B$ be a Banach space, $\lambda>0$ and $2 \leq q < \infty$. Then, the following statements are equivalent.
        \begin{itemize}
            \item[$(i)$] $\displaystyle \|S^q(f)\|_{L^1(\mathbb{R})} \leq C \|f\|_{H^1(\mathbb{R},\B)}, \quad f \in H^{1}(\mathbb{R},\B),$ \\
            \item[$(ii)$] $\displaystyle \|S^q_{\lambda,+}(f)\|_{L^1(0,\infty)} \leq C \|f\|_{H^1_o(\mathbb{R},\B)}, \quad f \in H^1_o(\mathbb{R},\B).$ \\ \quad
        \end{itemize}
    \end{Lem}

    \begin{proof}
        $(i) \Rightarrow (ii)$. We claim that the properties below are equivalent. \\
        \begin{itemize}
            \item[$(a)$] $\displaystyle \|S^q(f)\|_{L^1(\mathbb{R})} \leq C \|f\|_{H^1_o(\mathbb{R},\B)}, \quad f \in H^1_o(\mathbb{R},\B).$ \\
            \item[$(b)$] $\displaystyle \|S^q_+(f)\|_{L^1(0,\infty)} \leq C \|f\|_{H^1_o(\mathbb{R},\B)}, \quad f \in H^1_o(\mathbb{R},\B).$ \\
            \item[$(c)$] $\displaystyle \|S^q_{+,loc}(f)\|_{L^1(0,\infty)} \leq C \|f\|_{H^1_o(\mathbb{R},\B)}, \quad f \in H^1_o(\mathbb{R},\B).$ \\
            \item[$(d)$] $\displaystyle \|S^q_{\lambda,+}(f)\|_{L^1(0,\infty)} \leq C \|f\|_{H^1_o(\mathbb{R},\B)}, \quad f \in H^1_o(\mathbb{R},\B).$ \\ \quad
        \end{itemize}
        This implies what we are looking for.\\

        $(a) \Leftrightarrow (b)$. This equivalence can be proved by proceeding as in the proof of
        $(i) \Leftrightarrow (ii)$ of Lemma~\ref{Lem_Lp}.\\

        $(b) \Leftrightarrow (c)$. Let $f \in H^1_o(\mathbb{R},\B)$. Since $f$ is odd, Minkowski's inequality implies that
        \begin{align*}
            |S^q_+(f)(x)-S^q_{+,loc}(f)(x)|
                \leq \mathcal{H}_1(\|f\|_\B)(x) + \mathcal{H}_2(\|f\|_\B)(x), \quad x \in (0,\infty),
        \end{align*}
        where
        \begin{equation*}
            \mathcal{H}_1(g)(x) = \int_{(0,x/2) \cup (2x,\infty)} \left\| t \partial_t [P_t(y-z)-P_t(y+z)] \right\|_{L^q\left(\G_+(x),\frac{dtdy}{t^2}\right)} g(z) dz, \quad x \in (0,\infty),
        \end{equation*}
        and
        \begin{equation*}
            \mathcal{H}_2(g)(x) = \int_{x/2}^{2x} \left\| t \partial_t P_t(y+z) \right\|_{L^q\left(\G_+(x),\frac{dtdy}{t^2}\right)} g(z) dz, \quad x \in (0,\infty).
        \end{equation*}
        A straightforward manipulation leads to
        \begin{align*}
            \left| \partial_t[P_t(z-y)-P_t(z+y)] \right|
                \leq & C \left( \frac{zy}{((z-y)^2+t^2)((z+y)^2+t^2)} + \frac{zyt^2 (z^2+y^2+t^2)}{((z-y)^2+t^2)^2((z+y)^2+t^2)^2} \right) \\
                \leq & C \frac{zy}{((z-y)^2+t^2)((z+y)^2+t^2)}
                \leq   C \frac{\sqrt{z}}{(|z-y|+t)^2\sqrt{z+y+t}} \\
                \leq & C \frac{\sqrt{z}}{(|z-y|+t)^{5/2}}, \quad y,z,t \in (0,\infty).
        \end{align*}
        Hence, by proceeding as in \eqref{bound1} we obtain
        \begin{align*}
            \int_{\G_+(x)}  \left|t \partial_t[P_t(z-y)-P_t(z+y)] \right|^q \frac{dtdy}{t^2}
                \leq & C \int_{\G_+(x)} \frac{z^{q/2}t^{q-2}}{(|z-y|+t)^{5q/2}} dt dy \\
                \leq & C \frac{z^{q/2}}{|x-z|^{3q/2}}, \quad x,z \in (0,\infty), \ x \neq z.
        \end{align*}
        Then, we get
        $$\left\| t \partial_t [P_t(y-z)-P_t(y+z)] \right\|_{L^q\left(\G_+(x),\frac{dtdy}{t^2}\right)}
            \leq C \frac{\sqrt{z}}{|x-z|^{3/2}}, \quad x,z \in (0,\infty), \ x \neq z.$$
        Therefore,
        $$\| \mathcal{H}_1(g)  \|_{L^1(0,\infty)}
            \leq C \|g\|_{L^1(0,\infty)}, \quad g \in L^1(0,\infty).$$
        Moreover, according to \eqref{bound2}, we have that
        $$|\mathcal{H}_2(g)(x)|
            \leq C \int_{x/2}^{2x} \frac{|g(z)|}{z} dz, \quad x \in (0,\infty),$$
        and it follows that
        $$\| \mathcal{H}_2(g)  \|_{L^1(0,\infty)}
            \leq C \|g\|_{L^1(0,\infty)}, \quad g \in L^1(0,\infty).$$
        Hence, we conclude that
        $$ \| S^q_+(f)-S^q_{+,loc}(f) \|_{L^1(0,\infty)}
            \leq C \|f\|_{L^1((0,\infty),\B)}
            \leq C \|f\|_{H^1_o(\mathbb{R},\B)}.$$

        \quad

        $(c) \Leftrightarrow (d)$. Let $f \in H^1_o(\mathbb{R},\B)$. By \eqref{2.3a} we have to analyze the operators  $\mathcal{K}_j$, $j=1,2,3$.
        From \eqref{||Pt1||} it follows that
        \begin{align*}
            \int_0^\infty \mathcal{K}_1(\|f\|_\B)&(x) dx
                \leq  C \int_0^\infty \int_{(0,x/2) \cup (2x,\infty)} \frac{z^\lambda}{|x-z|^{\lambda+1}} \|f(z)\|_\B dz dx \\
                \leq & C \int_0^\infty \|f(z)\|_\B \left(\int_0^{z/2} \frac{z^\lambda}{|x-z|^{\lambda+1}} dx + \int_{2z}^\infty \frac{z^\lambda}{|x-z|^{\lambda+1}} dx\right) dz
                \leq  C \int_0^\infty \|f(z)\|_\B  dz,
        \end{align*}
        and from \eqref{||Pt2||} we deduce, in a similar way, that
        $$\int_0^\infty \mathcal{K}_2(\|f\|_\B)(x) dx
            \leq C \int_0^\infty \|f(z)\|_\B  dz.$$
        Finally, \eqref{2.8} implies that
        \begin{align*}
            \int_0^\infty \mathcal{K}_3(\|f\|_\B)(x) dx
                \leq & C \int_0^\infty \int_{x/2}^{2x} \frac{1}{z} \left( 1 + \log_+\frac{z}{|x-z|} \right) \|f(z)\|_\B dz dx \\
                \leq & C \int_0^\infty \|f(z)\|_\B \int_{z/2}^{2z} \frac{1}{z} \left( 1 + \log_+\frac{z}{|x-z|} \right) dx dz
                \leq  C \int_0^\infty \|f(z)\|_\B  dz.
        \end{align*}
        By combining the above estimates we get
        $$\|S^q_{+,loc}(f) - S^q_{\lambda,+}(f)\|_{L^1(0,\infty)}
            \leq C \|f\|_{L^1((0,\infty),\B)}
            \leq C \|f\|_{H^1_o(\mathbb{R},\B)}.$$
        Thus $(i) \Rightarrow (ii)$ is established.\\

        $(ii) \Rightarrow (i)$. Let $f \in H^1(\mathbb{R},\B)$. We write $f=f_o+f_e$, where
        $$f_o(x)=\frac{f(x)-f(-x)}{2}, \quad f_e(x)=\frac{f(x)+f(-x)}{2}, \quad x \in \mathbb{R}.$$
        It is clear that $f_o$, $f_e \in H^1(\mathbb{R},\B)$. Moreover
        $\|f_o\|_{H^1(\mathbb{R},\B)} \leq \|f\|_{H^1(\mathbb{R},\B)}$ and $\|f_e\|_{H^1(\mathbb{R},\B)} \leq \|f\|_{H^1(\mathbb{R},\B)}$.
        Assume that $(ii)$ holds. Since $f_o \in H^1_o(\mathbb{R},\B)$, by using that $(d) \Leftrightarrow (a)$, we get
        \begin{equation}\label{2.9}
            \|S^q(f_o)\|_{L^1(\mathbb{R})}
                \leq C \|f_o\|_{H^1(\mathbb{R},\B)}.
        \end{equation}
        We define $H^1_e(\mathbb{R},\B)$ the space constituted by  all those even functions in $H^1(\mathbb{R},\B)$. We consider the properties
        $(a')$, $(b')$, $(c')$ and $(d')$ that are analogous to $(a)$, $(b)$, $(c)$ and $(d)$, respectively, but replacing $H^1_o(\mathbb{R},\B)$
        by $H^1_e(\mathbb{R},\B)$. By proceeding as in the odd case we can see that $(a') \Leftrightarrow (b')$ and
        $(c') \Leftrightarrow (d')$. We are going to see that $(b') \Leftrightarrow (c')$.\\

        Suppose that $h \in H^1_e(\mathbb{R},\B)$ and that $h = \sum_{j=1}^\infty \lambda_j a_j$, where, for each $j \in \mathbb{N}$, $a_j$ is a
        $H^1(\mathbb{R},\B)$-atom, and $\{\lambda_j\}_{j=1}^\infty \subset \mathbb{C}$ is such that $\sum_{j=1}^\infty |\lambda_j|<\infty$. Then,
        $h=\sum_{j=1}^\infty \lambda_j (a_j+\tilde{a}_j)/2$, being $\tilde{a}_j(x)=a_j(-x)$, $x \in \mathbb{R}$, and $j \in \mathbb{N}$. We define
        $b_j$ and $\gamma_j$, $j \in \mathbb{N}$, as follows
        \begin{itemize}
            \item $b_j=a_j$ and $\gamma_j=\lambda_j/2$, provided that $\supp(a_j) \subset [0,\infty),$
            \item $b_j=\tilde{a}_j$ and $\gamma_j=\lambda_j/2$, when $\supp(a_j) \subset (-\infty,0],$
            \item if $\supp(a_j) \cap (0,\infty) \neq \varnothing$ and $\supp(a_j) \cap (-\infty,0) \neq \varnothing$, then
                  $b_j=\chi_{(0,\infty)}(a_j+\tilde{a}_j)/2$ and $\gamma_j=\lambda_j$.
        \end{itemize}
        Thus, $b_j$ is a $H^1(\mathbb{R},\B)$-atom. Indeed, in the two first cases it is clear. Assume now that $b_j=\chi_{(0,\infty)}(a_j+\tilde{a}_j)/2$
        where  $\supp(a_j) \cap (0,\infty) \neq \varnothing$, $\supp(a_j) \cap (-\infty,0) \neq \varnothing$, $\supp(a_j) \subset [-\alpha,\beta]$
        being $0<\alpha<\beta$ (similarly, $0<\beta<\alpha$), and $\|a_j\|_{L^\infty(\mathbb{R},\B)} \leq 1/(\beta+\alpha)$. Then, $\supp(b_j) \subset [0,\beta]$,
        $\|b_j\|_{L^\infty(\mathbb{R},\B)} \leq \|a_j\|_{L^\infty(\mathbb{R},\B)} \leq 1/(\beta+\alpha) \leq 1/\beta$, and
        $$\int_{\mathbb{R}} b_j(x) dx
            = \frac{1}{2}\int_{-\alpha}^\beta a_j(x) dx =0.$$

        From now on we write $h = \sum_{j=1}^\infty 2\gamma_j g_j$, where $g_j(x)=b_j(|x|)/2$, $x \in \mathbb{R}$ and $j \in \mathbb{N}$, and $b_j$
        and $\gamma_j$, $j \in \mathbb{N}$, are those ones we have just defined.\\

        We can write
        \begin{align*}
            & |S^q_+(h)(x)  -  S^q_{+,loc}(h)(x)|
                \leq  \left( \int_{\G_+(x)} \left\| t\partial_t \left(
                     \int_{\mathbb{R}} P_t(y-z) h(z) dz - \int_{x/2}^{2x}  P_t(y-z) h(z) dz
                     \right)\right\|^q_\B \frac{dtdy}{t^2} \right)^{1/q} \\
                & \quad  \leq  2\sum_{j=1}^\infty |\gamma_j| \left( \int_{\G_+(x)} \left\| t\partial_t \left(
                     \int_{\mathbb{R}} P_t(y-z) g_j(z) dz - \int_{x/2}^{2x}  P_t(y-z) g_j(z) dz
                     \right)\right\|^q_\B \frac{dtdy}{t^2} \right)^{1/q}, \ x \in (0,\infty).
        \end{align*}

        Assume that $g$ is an even $H^1(\mathbb{R},\B)$-atom such that $\supp(g) \subset [-\beta,\beta]$ with $\beta>0$ and $\|g\|_{L^\infty(\mathbb{R},\B)}\leq 1/2\beta$.
        We have that
        $$\int_{\mathbb{R}} P_t(y-z) g(z) dz
            = \int_0^\beta [P_t(y-z) + P_t(y+z)]g(z) dz, \quad y \in \mathbb{R}.$$
        For every $x \in (0,\infty)$, by using Minkowski's inequality we get
        \begin{align*}
            & \left( \int_{\G_+(x)} \left\| t\partial_t \left( \int_{\mathbb{R}} P_t(y-z) g(z) dz - \int_{x/2}^{2x}  P_t(y-z) g(z) dz
                                  \right)\right\|^q_\B \frac{dtdy}{t^2} \right)^{1/q}\\
                & \qquad = \left( \int_{\G_+(x)} \left\| t\partial_t \left( \int_{(0,x/2)\cup(2x,\infty)} P_t(y-z) g(z) dz + \int_{0}^\infty P_t(y+z) g(z) dz
                                                 \right)\right\|^q_\B \frac{dtdy}{t^2} \right)^{1/q} \\
                & \qquad\leq  \chi_{(0,2\beta)}(x) \int_0^{x/2} \left\| t \partial_t P_t(y-z) \right\|_{L^q\left(\G_+(x),\frac{dtdy}{t^2}\right)} \|g(z)\|_\B dz \\
                & \qquad \qquad + \chi_{(2\beta, \infty)}(x) \int_0^{\beta}   \left\|t \partial_t [P_t(y-z)-P_t(y+\beta)]   \right\|_{L^q\left(\G_+(x),\frac{dtdy}{t^2}\right)} \|g(z)\|_\B dz\\
                & \qquad \qquad + \int_{2x}^\infty \left\| t \partial_t P_t(y-z) \right\|_{L^q\left(\G_+(x),\frac{dtdy}{t^2}\right)} \|g(z)\|_\B dz \\
                & \qquad \qquad + \chi_{(0,2\beta)}(x) \int_0^{x/2} \left\| t \partial_t P_t(y+z) \right\|_{L^q\left(\G_+(x),\frac{dtdy}{t^2}\right)} \|g(z)\|_\B dz \\
                & \qquad \qquad  + \chi_{(2\beta, \infty)}(x)  \int_0^{\beta}  \left\|t \partial_t [P_t(y+z)-P_t(y+\beta)]   \right\|_{L^q\left(\G_+(x),\frac{dtdy}{t^2}\right)} \|g(z)\|_\B dz\\
                & \qquad \qquad  + \int_{x/2}^\infty \left\| t \partial_t P_t(y+z) \right\|_{L^q\left(\G_+(x),\frac{dtdy}{t^2}\right)} \|g(z)\|_\B dz \\
            &\qquad =      \sum_{i=1}^6 \mathpzc{H}_{i}(\|g\|_\B)(x).
        \end{align*}
        Note that it is possible to introduce the factor $t \partial_t P_t(y+\beta)$, because $g$ is even and has zero mean.\\

        Our goal is to see that, for a certain $C>0$ independent of $g$,
        \begin{equation*} %\label{less1 H}
            \| \mathpzc{H}_{i}(\|g\|_\B) \|_{L^1(0,\infty)}
                \leq C, \quad i=1, \dots, 6.
        \end{equation*}

        According to \eqref{bound1} it follows that
        \begin{align*}
            \mathpzc{H}_{1}(\|g\|_\B)(x)
                \leq & C \chi_{(0,2\beta)}(x) \int_0^{x/2} \frac{\|g(z)\|_\B}{x-z} dz
                \leq C \frac{\chi_{(0,2\beta)}(x)}{\beta}, \quad x \in (0,\infty),
        \end{align*}
        and then
        \begin{equation}\label{2.10}
            \| \mathpzc{H}_{1}(\|g\|_\B) \|_{L^1(0,\infty)}
                \leq C.
        \end{equation}
        In a similar way \eqref{bound2} leads to
        \begin{equation}\label{2.11}
            \| \mathpzc{H}_{4}(\|g\|_\B) \|_{L^1(0,\infty)}
                \leq C.
        \end{equation}

        By using again \eqref{bound1} and \eqref{bound2} we obtain
        \begin{align*}
            \mathpzc{H}_{3}(\|g\|_\B)(x) + \mathpzc{H}_{6}(\|g\|_\B)(x)
                \leq & C  \int_{x/2}^\infty \frac{\|g(z)\|_\B}{z} dz, \quad x \in (0,\infty).
        \end{align*}
        Since the Hardy operator $H_\infty$ is bounded from $L^1(0,\infty)$ into itself, we conclude that
        \begin{equation}\label{2.12}
            \| \mathpzc{H}_{3}(\|g\|_\B) \|_{L^1(0,\infty)} + \| \mathpzc{H}_{6}(\|g\|_\B) \|_{L^1(0,\infty)}
                \leq C \int_0^\beta \|g(z)\|_\B dz
                \leq C.
        \end{equation}

        In order to analyze $\mathpzc{H}_{j}(\|g\|_\B)$, $j=2,5$, we claim that
        \begin{equation}\label{y+beta}
            \left\|t \partial_t [P_t(y \pm z)-P_t(y+\beta)]   \right\|_{L^q\left(\G_+(x),\frac{dtdy}{t^2}\right)}
                \leq C \frac{\beta}{|x-z|^2}, \quad x \in (0,\infty), \ 0<z<\beta \text{ and } x \neq z.
        \end{equation}
        If \eqref{y+beta} holds we obtain
        \begin{align}\label{2.14}
            \| \mathpzc{H}_{2}(\|g\|_\B) \|_{L^1(0,\infty)} + \| \mathpzc{H}_{5}(\|g\|_\B) \|_{L^1(0,\infty)}
                \leq & C  \int_0^\beta \|g(z)\|_\B \int_{2\beta}^\infty \frac{\beta}{|x-z|^2}  dx dz \nonumber \\
                \leq & C \int_0^\beta \|g(z)\|_\B  \frac{\beta}{2\beta-z} dz
                \leq C.
        \end{align}
        Note that the constants $C>0$ in \eqref{2.10}-\eqref{2.14} do not depend on $g$.\\

        To justify \eqref{y+beta} we observe that
        \begin{align*}
            \partial_t [P_t(y - z)-P_t(y+\beta)]
                 = & \frac{\beta^2-z^2+2y(\beta+z)}{\pi} \left( \frac{1}{(t^2+|y-z|^2)(t^2+(y+\beta)^2)} \right. \\
                 & \left. - 2t^2 \frac{2t^2+(y+\beta)^2+|y-z|^2}{(t^2+|y-z|^2)^2(t^2+(y+\beta)^2)^2} \right), \quad y,z \in \mathbb{R}, \ t>0.
        \end{align*}
        Moreover, if $0<y<\infty$ and $0<z<\beta$, $|y-z| \leq y + \beta$, and
        \begin{align*}
            |\partial_t [P_t(y - z)-P_t(y+\beta)]|
                \leq C & \frac{\beta(y+\beta)}{(t+|y-z|)^2(t+y+\beta)^2}
                \leq \frac{\beta}{(t+|y-z|)^3}.
        \end{align*}
        In a similar way we can see, for each $0<y<\infty$ and $0<z<\beta$,
        \begin{align*}
            |\partial_t [P_t(y + z)-P_t(y+\beta)]|
                \leq C \frac{\beta}{(t+y+z)^2(t+y+\beta)}
                \leq  \frac{\beta}{(t+|y-z|)^3}.
        \end{align*}
        By proceeding as in \eqref{bound1} we obtain \eqref{y+beta}.\\

        Putting together all the above estimations we conclude that
        \begin{align*}
            \|S^q_+(h) - S^q_{+,loc}(h)\|_{L^1(0,\infty)}
                \leq & C \sum_{j=1}^\infty |\gamma_j|
                \leq  C \sum_{j=1}^\infty |\lambda_j|.
        \end{align*}
        Hence,
        \begin{align*}
            \|S^q_+(h) - S^q_{+,loc}(h)\|_{L^1(0,\infty)}
                \leq & C \|h\|_{H^1(\mathbb{R},\B)}.
        \end{align*}
        Thus, $(b') \Leftrightarrow (c')$ is established.\\

        Assume again that $h \in H^1_e(\mathbb{R},\B)$. We define $H$ as the odd extension of $h_{|_{[0,\infty)}}$ to $\mathbb{R}$. It is clear
        that $H \in H^1_o(\mathbb{R},\B)$ and $\|H\|_{H^1(\mathbb{R},\B)} \leq C \|h\|_{H^1(\mathbb{R},\B)}$. Then, according to $(ii)$ we get
        \begin{align*}
            \|S^q_{\lambda,+}(h)\|_{L^1(0,\infty)}
                = & \|S^q_{\lambda,+}(H)\|_{L^1(0,\infty)}
                \leq C \|H\|_{H^1(\mathbb{R},\B)}
                \leq C \|h\|_{H^1(\mathbb{R},\B)}.
        \end{align*}
        Hence, we have that
        \begin{equation}\label{2.15}
            \|S^q(f_e)\|_{L^1(\mathbb{R})}
                \leq C \|f_e\|_{H^1(\mathbb{R},\B)}.
        \end{equation}
        By combining \eqref{2.9} and \eqref{2.15} we conclude that
        \begin{equation*}
            \|S^q(f)\|_{L^1(\mathbb{R})}
                \leq C \|f\|_{H^1(\mathbb{R},\B)}.
        \end{equation*}
        Thus the proof of this lemma is completed.
    \end{proof}

    The proof of Proposition~\ref{Lem_principal} is now consequence of Lemmas~\ref{Lem_Lp} and \ref{Lem_L1} and \cite[Lemma 4.2]{OX}.

%\newpage
    %%%%%%%%%%%%%%%%%%%%%%%%%%%%%%%%%%%%%%%%%%%%%%%%%%%%%%%%%%%%%%%%%%%%%%%%%%%%%%%%%%%%%%%%%%%%%%%%%%%%%%%%%%%%%%%%%%%%
    \section{Proof of Theorem~\ref{Th_3.1_OX}} \label{sec:proof1}
    %%%%%%%%%%%%%%%%%%%%%%%%%%%%%%%%%%%%%%%%%%%%%%%%%%%%%%%%%%%%%%%%%%%%%%%%%%%%%%%%%%%%%%%%%%%%%%%%%%%%%%%%%%%%%%%%%%%%

    \subsection{Proof of $(ii) \Rightarrow (i)$} \label{subsec:2->1}

    Assume that $\B$ has an equivalent norm which is $q$-uniformly convex.\\

    Let $f \in BMO_o(\mathbb{R},\B)$ and take $I=(a,b)$ such that $0 \leq a<b<\infty$. We denote by $2I$
    the interval $(0,\infty) \cap (x_I-|I|,x_I+|I|)$ where $x_I=(a+b)/2$. We decompose $f$ as follows:
    $$f \chi_{(0,\infty)}=(f-f_{2I})\chi_{2I} + (f-f_{2I}) \chi_{(0,\infty)\setminus 2I} + f_{2I}=f_1  + f_2 + f_3.$$
    We are going to show, for $i=1,2,3,$
    \begin{equation}\label{objetivo2}
        \left( \frac{1}{|I|}\int_0^{|I|} \int_I \|t \partial_t P_t^\lambda(f_i)(x) \|^q_\B \frac{dx dt}{t} \right)^{1/q}
                                        \leq C \|f\|_{BMO_o(\mathbb{R},\B)},
    \end{equation}
    where the constant $C>0$ depends neither on $I$ nor on $f$.\\

    Firstly, we prove \eqref{objetivo2} for $i=1$. Note that $|I \cap (x-t,x+t)| \sim t$, when $x \in I$ and $0<t<|I|$. We have that
    \begin{align} \label{3.1.a}
        \int_0^{|I|} \int_I \|t \partial_t P_t^\lambda(f_1)(x) \|^q_\B \frac{dx dt}{t}
            \leq & C \int_0^{|I|} \int_I \|t \partial_t P_t^\lambda(f_1)(x) \|^q_\B \frac{|I \cap (x-t,x+t)|}{t^2} dx dt  \nonumber\\
            = & C \int_0^{|I|} \int_I \int_{I \cap (x-t,x+t)} \|t \partial_t P_t^\lambda(f_1)(x) \|^q_\B dz \frac{dx dt}{t^2} \nonumber \\
            \leq & C \int_I \int_{\G_+(z)} \|t \partial_t P_t^\lambda(f_1)(x) \|^q_\B  \frac{dt dx }{t^2}dz
            \leq  C \|S^q_{\lambda,+}(f_1)\|_{L^q(0,\infty)}^q.
    \end{align}
    By using Proposition~\ref{Lem_principal} and John-Nirenberg's inequality we get
    \begin{align*}
        \left( \frac{1}{|I|}\int_0^{|I|} \int_I \|t \partial_t P_t^\lambda(f_1)(x) \|^q_\B \frac{dx dt}{t} \right)^{1/q}
            \leq & C \left( \frac{1}{|I|} \|S^q_{\lambda,+}(f_1)\|_{L^q(0,\infty)}^q \right)^{1/q} \\
            \leq & C \left( \frac{1}{|I|} \int_{2I} \| f(x)  - f_{2I}\|_\B^q dx\right)^{1/q}
            \leq C \|f\|_{BMO_o(\mathbb{R},\B)}.
    \end{align*}

    On the other hand, from \cite[(b), p.  86]{MS} we deduce that
    \begin{align}\label{3.2a}
        |t\partial_t P_t^\lambda(x,y)|
            \leq & C \left( P_t^\lambda(x,y) + t^3 (xy)^\lambda \int_0^\pi \frac{(\sin \theta)^{2\lambda-1}}{(x^2 + y^2 + t^2 - 2xy \cos \theta)^{\lambda+2}} d\theta \right) \nonumber \\
            \leq & C P_t^\lambda(x,y)
            \leq C \frac{t}{t^2+(x-y)^2}, \quad x,y,t \in (0,\infty).
    \end{align}
    Then, we can write, for every $x \in I$ and $t>0$,
    \begin{align*}
        \| t\partial_t P_t^\lambda(f_2)(x) \|_\B
            \leq & C \int_{(0,\infty) \setminus 2I} \frac{t}{t^2+(x-y)^2} \|f(y)-f_{2I}\|_\B dy \\
            \leq & C \int_{(0,\infty) \setminus 2I} \frac{t}{t^2+(x_I-y)^2} \|f(y)-f_{2I}\|_\B dy \\
            \leq & C \frac{t}{|I|} \sum_{k=1}^\infty \frac{1}{2^k} \left( \frac{1}{2^k |I|} \int_{2^{k+1}I \cap (0,\infty)}  \|f(y)-f_{2^{k+1}I}\|_\B dy + \|f_{2^{k+1}I} - f_{2I}\|_\B \right) \\
            \leq & C \frac{t}{|I|} \sum_{k=1}^\infty \frac{k}{2^k} \|f\|_{BMO_o(\mathbb{R},\B)}
            \leq C \frac{t}{|I|} \|f\|_{BMO_o(\mathbb{R},\B)}.
    \end{align*}
    Hence,
    \begin{align*}
        \left( \frac{1}{|I|}\int_0^{|I|} \int_I \|t \partial_t P_t^\lambda(f_2)(x) \|^q_\B \frac{dx dt}{t} \right)^{1/q}
            \leq & C \left( \frac{1}{|I|^{q+1}}\int_0^{|I|} t^{q-1} dt \int_I dx  \right)^{1/q} \|f\|_{BMO_o(\mathbb{R},\B)} \\
            \leq & C \|f\|_{BMO_o(\mathbb{R},\B)}.
    \end{align*}

    Finally, we show \eqref{objetivo2} for $i=3$. Observe that in the classical case (see \cite{OX}) this term does not appear because the
    corresponding integral vanishes. First of all, we notice that
    $$\|f_{2I}\|_\B \leq \frac{1}{2|I|} \int_0^{x_I+|I|} \|f(x)\|_\B dx \leq \frac{x_I+|I|}{2|I|} \|f\|_{BMO_o(\mathbb{R},\B)}.$$
    Then, in order to establish \eqref{objetivo2} for $i=3$ it is sufficient to show that
    \begin{equation}\label{objetivo3}
          \frac{(x_I+|I|)^q}{|I|^{q+1}}\int_0^{|I|} \int_I  |t\partial_t P_t^\lambda(1)(x)|^q \frac{dx dt}{t}
            \leq C,
    \end{equation}
    where $C>0$ does not depend on $I$.\\

    By taking into account that $|t \partial_t P_t^\lambda(x,y)| \leq C P_t^\lambda(x,y)$, $x,y,t \in (0,\infty)$,
    (see \eqref{3.2a}) we can write
    \begin{align*}
         |t\partial_t P_t^\lambda(1)(x)|
            \leq & C \left( \int_0^{x/2} P_t^\lambda(x,y) dy +  \left| t\partial_t \int_{x/2}^{3x/2}  P_t^\lambda(x,y) dy \right| + \int_{3x/2}^\infty P_t^\lambda(x,y) dy \right)
            =  C \sum_{j=1}^3 J_j(x,t).
    \end{align*}
    According to \cite[(b), p. 86]{MS} we have that,
    \begin{equation}\label{J1}
        J_1(x,t)
            \leq C \int_0^{x/2} \frac{(xy)^\lambda t}{[(x-y)^2+t^2]^{\lambda+1}} dy
            \leq C  \frac{x^{2\lambda+1} t}{(x^2+t^2)^{\lambda+1}}
            \leq C \frac{t}{t+x}, \quad x,t \in (0,\infty),
    \end{equation}
    \begin{equation}\label{J2}
        J_2(x,t)
            \leq C \int_{x/2}^{3x/2} \frac{(xy)^\lambda t}{[(x-y)^2+t^2]^{\lambda+1}} dy
            \leq C  \frac{x^{2\lambda+1}}{t^{2\lambda+1}}, \quad x,t \in (0,\infty),
    \end{equation}
    and
    \begin{align}\label{J3}
        J_3(x,t)
            \leq & C \int_{3x/2}^\infty \frac{(xy)^\lambda t}{[(x-y)^2+t^2]^{\lambda+1}} dy
            \leq  C x^\lambda t \int_{3x/2}^\infty \frac{y^\lambda }{(y^2+t^2)^{\lambda+1}} dy \nonumber \\
            \leq & C x^\lambda t \int_{3x/2}^\infty \frac{1 }{(y+t)^{\lambda+2}} dy
            \leq C \frac{x^\lambda t}{(x+t)^{\lambda+1}}
            \leq C \frac{t}{t+x}, \quad x,t \in (0,\infty).
    \end{align}

    We need also to estimate $J_2$ in a different way. The classical Poisson kernel is introduced as follows
    \begin{align*}
        J_2(x,t)
            \leq &  \left|\int_{x/2}^{3x/2} t\partial_t [P_t^\lambda(x,y) - P_t(x-y)]  dy \right|  + \left| t\partial_t\int_{x/2}^{3x/2}  P_t(x-y) dy \right|
            =  \sum_{j=1}^2 J_{2,j}(x,t), \ x,t \in (0,\infty).
    \end{align*}
    The function under the integral sign in $J_{2,1}$ is decomposed as follows
    \begin{align}\label{3.5}
        t\partial_t & [P_t^\lambda(x,y) - P_t(x-y)]
            =  P_t^\lambda(x,y)- P_t(x-y) \nonumber \\
              &  - \left( \frac{4(xy)^\lambda t^3 \lambda (\lambda+1)}{\pi} \int_0^\pi \frac{(\sin \theta)^{2\lambda-1}}{((x-y)^2 + t^2 + 2xy(1- \cos \theta))^{\lambda+2}} d\theta
                - \frac{2t^3}{\pi} \frac{1}{((x-y)^2+t^2)^2}\right) \nonumber \\
            = & (P_{t,1}^\lambda(x,y)- P_t(x-y)) + P_{t,2}^\lambda(x,y)  \nonumber \\
              &  - \left( \frac{4(xy)^\lambda t^3 \lambda (\lambda+1)}{\pi} \int_0^{\pi/2} \frac{(\sin \theta)^{2\lambda-1}}{((x-y)^2 + t^2 + 2xy(1- \cos \theta))^{\lambda+2}} d\theta
                - \frac{2t^3}{\pi} \frac{1}{((x-y)^2+t^2)^2}\right) \nonumber \\
              & - \frac{4(xy)^\lambda t^3 \lambda (\lambda+1)}{\pi} \int_{\pi/2}^\pi \frac{(\sin \theta)^{2\lambda-1}}{(x^2+y^2 + t^2 - 2xy\cos \theta)^{\lambda+2}} d\theta,
              \quad x,y,t \in (0,\infty),
    \end{align}
    where as above $P_{t,1}^\lambda$ and $P_{t,2}^\lambda$ is defined in \eqref{7.2}.\\

    Firstly note that
    \begin{align}\label{3.6}
        \Big| (xy)^\lambda t^3   &  \int_{\pi/2}^\pi \frac{(\sin \theta)^{2\lambda-1}}{(x^2 + y^2 + t^2 - 2xy \cos \theta)^{\lambda+2}} d\theta \Big|
            \leq  C P_{t,2}^\lambda(x,y) \nonumber \\
            \leq & C  \frac{(xy)^\lambda t}{(x+y+t)^{2\lambda+2}} \int_{\pi/2}^\pi (\sin \theta)^{2\lambda-1} d\theta
            \leq  C \frac{t}{xy}, \quad x,y,t \in (0,\infty).
    \end{align}

    Since $\sin \theta \sim \theta$ and $2(1-\cos \theta) \sim \theta^2$, $\theta \in [0,\pi/2]$, by using the mean value theorem we obtain
    \begin{align}\label{3.7}
        & \left| \frac{2(xy)^\lambda t^3 \lambda (\lambda+1)}{\pi} \int_0^{\pi/2} \frac{(\sin \theta)^{2\lambda-1}}{((x-y)^2 + t^2 + 2xy(1- \cos \theta))^{\lambda+2}} d\theta
                - \frac{t^3}{\pi} \frac{1}{((x-y)^2+t^2)^2}\right| \nonumber \\
        & \quad \leq \left| \frac{2(xy)^\lambda t^3 \lambda (\lambda+1)}{\pi} \int_0^{\pi/2} \frac{(\sin \theta)^{2\lambda-1}-\theta^{2\lambda-1}}{((x-y)^2 + t^2 + 2xy(1- \cos \theta))^{\lambda+2}} d\theta \right| \nonumber \\
        & \qquad + \left| \frac{2(xy)^\lambda t^3 \lambda (\lambda+1)}{\pi} \int_0^{\pi/2}\theta^{2\lambda-1}
                                \left( \frac{1}{((x-y)^2 + t^2 + 2xy(1- \cos \theta))^{\lambda+2}} - \frac{1}{((x-y)^2 + t^2 + xy \theta^2)^{\lambda+2}}\right) d\theta \right| \nonumber \\
        & \qquad + \left| \frac{2(xy)^\lambda t^3 \lambda (\lambda+1)}{\pi} \int_0^{\pi/2} \frac{\theta^{2\lambda-1}}{((x-y)^2 + t^2 + xy \theta^2)^{\lambda+2}} d\theta
                - \frac{t^3}{\pi} \frac{1}{((x-y)^2+t^2)^2}\right| \nonumber \\
        & \quad \leq C \left( (xy)^\lambda t^3 \int_0^{\pi/2} \frac{\theta^{2\lambda+1}}{((x-y)^2 + t^2 + xy\theta^2)^{\lambda+2}} d\theta
                            + (xy)^{\lambda+1} t^3 \int_0^{\pi/2} \frac{\theta^{2\lambda+3}}{((x-y)^2 + t^2 + xy\theta^2)^{\lambda+3}} d\theta \right) \nonumber\\
        & \qquad + \frac{t^3}{\pi}\left| \frac{2 \lambda (\lambda+1)}{((x-y)^2+t^2)^2} \int_0^{\frac{\pi}{2}\sqrt{\frac{xy}{(x-y)^2+t^2}}} \frac{u^{2\lambda-1}}{(1+u^2)^{\lambda+2}} du
                - \frac{1}{((x-y)^2+t^2)^2}\right| \nonumber
    \end{align}
    \begin{align}
        & \quad \leq C t^3 \left( \frac{1}{xy((x-y)^2+t^2)} \int_0^{\frac{\pi}{2}\sqrt{\frac{xy}{(x-y)^2+t^2}}} \frac{u^{2\lambda+1}}{(1+u^2)^{\lambda+2}} du
                                + \frac{1}{((x-y)^2+t^2)^2} \int_{\frac{\pi}{2}\sqrt{\frac{xy}{(x-y)^2+t^2}}}^\infty \frac{u^{2\lambda-1}}{(1+u^2)^{\lambda+2}} du\right) \nonumber \\
        & \quad \leq C t^3 \left( \frac{1}{xy((x-y)^2+t^2)} \int_0^\infty \frac{u^{2\lambda+1}}{(1+u^2)^{\lambda+2}} du
                                + \frac{1}{((x-y)^2+t^2)^2} \frac{1}{1+\frac{xy}{(x-y)^2+t^2}}\int_{0}^\infty \frac{u^{2\lambda-1}}{(1+u^2)^{\lambda+1}} du\right) \nonumber\\
        & \quad \leq C \frac{t}{xy}, \quad x,y,t \in (0,\infty).
    \end{align}
    We have used that $  \int_0^\infty u^{2\lambda-1}/(1+u^2)^{\lambda+2} du = 1/(2\lambda(\lambda+1))$.\\

    Finally, by proceeding in a similar way and using that $  \int_0^\infty u^{2\lambda-1}/(1+u^2)^{\lambda+1} du = 1/2\lambda$,
    we get
    \begin{align}\label{3.8}
        & |P_{t,1}^\lambda(x,y)- P_t(x-y)|
        = \left| \frac{2(xy)^\lambda t \lambda }{\pi} \int_0^{\pi/2} \frac{(\sin \theta)^{2\lambda-1}}{((x-y)^2 + t^2 + 2xy(1- \cos \theta))^{\lambda+1}} d\theta
                - \frac{t}{\pi} \frac{1}{(x-y)^2+t^2}\right| \nonumber \\
        & \quad \leq \left| \frac{2(xy)^\lambda t\lambda }{\pi} \int_0^{\pi/2} \frac{(\sin \theta)^{2\lambda-1}-\theta^{2\lambda-1}}{((x-y)^2 + t^2 + 2xy(1- \cos \theta))^{\lambda+1}} d\theta \right| \nonumber \\
        & \qquad + \left| \frac{2(xy)^\lambda t \lambda }{\pi} \int_0^{\pi/2}\theta^{2\lambda-1}
                                \left( \frac{1}{((x-y)^2 + t^2 + 2xy(1- \cos \theta))^{\lambda+1}} - \frac{1}{((x-y)^2 + t^2 + xy \theta^2)^{\lambda+1}}\right) d\theta \right| \nonumber \\
        & \qquad + \left| \frac{2(xy)^\lambda t \lambda }{\pi} \int_0^{\pi/2} \frac{\theta^{2\lambda-1}}{((x-y)^2 + t^2 + xy \theta^2)^{\lambda+1}} d\theta
                - \frac{t}{\pi} \frac{1}{(x-y)^2+t^2}\right| \nonumber \\
        & \quad \leq C \left( (xy)^\lambda t \int_0^{\pi/2} \frac{\theta^{2\lambda+1}}{((x-y)^2 + t^2 + xy\theta^2)^{\lambda+1}} d\theta
                            + (xy)^{\lambda+1} t \int_0^{\pi/2} \frac{\theta^{2\lambda+3}}{((x-y)^2 + t^2 + xy\theta^2)^{\lambda+2}} d\theta  \right. \nonumber\\
        & \qquad \left. + t \left| \frac{2 \lambda }{(x-y)^2+t^2} \int_0^{\frac{\pi}{2}\sqrt{\frac{xy}{(x-y)^2+t^2}}} \frac{u^{2\lambda-1}}{(1+u^2)^{\lambda+1}} du
                - \frac{1}{(x-y)^2+t^2}\right| \right) \nonumber\\
        & \quad \leq C t \left( \int_0^{\pi/2} \frac{\theta}{(x-y)^2 + t^2 + xy\theta^2} d\theta +
            \frac{1}{(x-y)^2+t^2}\int_{\frac{\pi}{2}\sqrt{\frac{xy}{(x-y)^2+t^2}}}^\infty \frac{u^{2\lambda-1}}{(1+u^2)^{\lambda+1}} du\right) \nonumber\\
        & \quad \leq C \frac{t}{xy} \left(1+ \log \left( 1 + \frac{xy}{(x-y)^2} \right) \right)
            , \quad x,y,t \in (0,\infty), \quad x \neq y.
    \end{align}
    Putting together \eqref{3.5}-\eqref{3.8} we obtain
    $$J_{2,1}(x,t)
        \leq C \frac{t}{x} \int_{x/2}^{3x/2} \frac{1}{y} \left( 1+\log \left( 1 + \frac{x^2}{(x-y)^2} \right)  \right) dy
        \leq C \frac{t}{x}, \quad x,t \in (0,\infty).$$
    Moreover, we have that
    \begin{align*} % \label{J22}
        J_{2,2}(x,t)
            &= \left |t\partial_t \int_{x/2}^{3x/2} P_t(x-y) dy \right|
            = \frac{2}{\pi}  \left |t\partial_t \int_{0}^{x/2} \frac{t}{t^2+u^2} du \right|
                      = \frac{2}{\pi}  \left |t\partial_t \left( \frac{\pi}{2} - \int_{x/2}^\infty \frac{t}{t^2+u^2} du \right) \right| \nonumber  \\
                     & \leq C t \int_{x/2}^\infty \frac{du}{u^2} \leq C \frac{t}{x}, \quad x,t \in (0,\infty).
    \end{align*}
    Hence, it follows that
    \begin{equation}\label{3.9}
        J_2(x,t)
            \leq C \frac{t}{x}, \quad x,t \in (0,\infty).
    \end{equation}

    We now prove \eqref{objetivo3}. Suppose firstly $|I| \leq x_I$. Then, since $q \geq 2$,
    \begin{align*}
        \frac{(x_I+|I|)^q}{|I|^{q+1}}\int_0^{|I|} \int_I  |t\partial_t P_t^\lambda(1)(x)|^q \frac{dx dt}{t}
        \leq &C  \frac{x_I^q}{|I|^{q+1}}\int_0^{|I|} t^{q-1} dt \int_{x_I-|I|/2}^{x_I+|I|/2}  \frac{dx}{x^q}\\
        \leq &C  \frac{x_I^q}{(x_I-|I|/2)^{q}}
        \leq C,
    \end{align*}
    because $|t\partial_t P_t^\lambda(1)(x)| \leq C t/x$, $x,t \in (0,\infty)$ (see \eqref{J1}, \eqref{J3} and \eqref{3.9}).\\

    Assume now that $|I|>x_I$. From \eqref{J1} and \eqref{J3} we deduce
    \begin{align*}
        \frac{(x_I+|I|)^q}{|I|^{q+1}}\int_0^{|I|} \int_I  J_i(x,t)^q \frac{dx dt}{t}
            \leq  & \frac{C}{|I|} \int_0^{|I|} t^{q-1} \int_0^{x_I+|I|} \frac{dx}{(t+x)^q} dt
            \leq C, \quad i=1,3.
    \end{align*}
    Finally, from \eqref{J2} and \eqref{3.9} it follows that
    \begin{align*}
        \frac{(x_I+|I|)^q}{|I|^{q+1}}\int_0^{|I|} \int_I  J_2(x,t)^q \frac{dx dt}{t}
            \leq  &   \frac{C}{|I|}  \int_I  \left( \int_0^x \left( \frac{t}{x} \right)^q \frac{dt}{t} +  \int_x^\infty \left( \frac{x}{t} \right)^{(2\lambda+1)q} \frac{dt}{t}\right) dx
            \leq C.
    \end{align*}

    Hence, \eqref{objetivo3} is established and the proof of $(ii) \Rightarrow (i)$ is finished.

    \begin{flushright}
        \qed
    \end{flushright}

%\newpage
    \subsection{Proof of $(i) \Rightarrow (ii)$} \label{subsec:1->2}

    Assume that $(i)$ holds. According to Proposition~\ref{Lem_principal}, in order to see that $\B$ has an equivalent $q$-uniformly convex norm
    it is enough to prove that there exists $C>0$ such that
    \begin{equation}\label{16.1}
        \| S^q_{\lambda,+}(f) \|_{L^q(0,\infty)}
            \leq C \|f\|_{L^q((0,\infty),\B)}, \quad f \in L^q((0,\infty),\B).
    \end{equation}
    Note firstly that \eqref{16.1} is a finite dimensional inequality in the following sense: if \eqref{16.1} holds when $\B$ is replaced
    by $\E$, where $\E$ is a finite dimensional subspace of $\B$, with a constant $C>0$ independent of $\E$, then \eqref{16.1}
    is also true for every $f \in L^q((0,\infty),\B)$ with the same constant $C>0$. This fact is a consequence of the density of
    $L^q(0,\infty) \otimes \B$ into $L^q((0,\infty),\B)$. Recall that every $f \in L^q(0,\infty) \otimes \B$ can be written as $f=\sum_{j=1}^n b_j f_j$, where $b_j \in \B$, $f_j \in L^q(0,\infty)$, $j=1, \dots, n$, and $n \in \mathbb{N}$.\\

    Let $\E$ be a subspace of $\B$ such that $\dim \E<\infty$. Applying again Proposition~\ref{Lem_principal}, instead of proving \eqref{16.1}
    for functions taking values in $\E$, it is sufficient to show that
    \begin{equation}\label{3.10}
        \| S^q_{\lambda,+}(f) \|_{L^1(0,\infty)}
            \leq C \|f\|_{H^1_o(\mathbb{R},\E)}, \quad f \in H^1_o(\mathbb{R},\E),
    \end{equation}
    being $C>0$ a constant independent of $\E$.
    Moreover, \eqref{3.10} holds provided that, for a certain $C>0$,
    \begin{equation}\label{3.11}
        \| S^q_{\lambda}(f) \|_{BMO(0,\infty)}
            \leq C \|f\|_{L^\infty((0,\infty),\E)}, \quad f \in L^\infty_c((0,\infty),\E),
    \end{equation}
    where $L^\infty_c((0,\infty),\E)$ denotes the space of functions in $L^\infty((0,\infty),\E)$ that have compact support. To make easier the reading of this part, the proof of that \eqref{3.11} implies \eqref{3.10} will be included in Section~\ref{sec:appendix} (see Proposition~\ref{Prop_BMO_H1}).\\

    Observe that \eqref{3.11} can be written as follows
    \begin{equation}\label{objetivo4}
        \left\| t \partial_t P_t^\lambda(f)(x+y) \right\|_{BMO\left((0,\infty),L^q\left(\G(0),\frac{dtdy}{t^2},\E\right)\right)}
            \leq C \|f\|_{L^\infty((0,\infty),\E)}, \quad f \in L^\infty_{c}((0,\infty),\E).
    \end{equation}

    Inequality \eqref{objetivo4} will be proved by using duality. Our objective is to show that, there exists $C>0$ such that for every
    $f \in L^\infty_{c}((0,\infty),\E)$ and  $h \in H^1\left((0,\infty), L^{q'}\left(\G(0),\frac{dtdy}{t^2}, \E^*\right) \right)$,
    \begin{equation}\label{3.13}
        |\langle t \partial_t P_t^\lambda(f)(x+y), h(x,y,t) \rangle|
            \leq C  \|f\|_{L^\infty((0,\infty),\E)} \|h\|_{H^1\left((0,\infty), L^{q'}\left(\G(0),\frac{dtdy}{t^2}, \E^*\right) \right)}.
    \end{equation}
    Recalling the atomic definition of $H^1\left((0,\infty), L^{q'}\left(\G(0),\frac{dtdy}{t^2}, \E^*\right) \right)$,
    by density arguments it is sufficient to prove \eqref{3.13} for every $f \in L^\infty_{c}((0,\infty),\E)$ and
    $h \in L^\infty_{c}\left((0,\infty), L^{q'}\left(\G(0),\frac{dtdy}{t^2},\E^*\right) \right)
           \cap H^1\left((0,\infty), L^{q'}\left(\G(0),\frac{dtdy}{t^2},\E^*\right) \right) $.\\

    Let $f \in L^\infty_{c}((0,\infty),\E)$ and $h \in L^\infty_{c}\left((0,\infty), L^{q'}\left(\G(0),\frac{dtdy}{t^2},\E^*\right) \right)$.
    We can write,
    \begin{align*}
        \langle t \partial_t P_t^\lambda(f)(x+y), h(x,y,t) \rangle
            =& \int_0^\infty \int_{\G(0)} \langle t \partial_t P_t^\lambda(f)(x+y),h(x,y,t) \rangle \frac{dtdy}{t^2} dx\\
            =& \lim_{N \to \infty} \int_0^\infty \int_{\G_N(0)} \langle t \partial_t P_t^\lambda(f)(x+y),h(x,y,t)\rangle \frac{dtdy}{t^2} dx,
    \end{align*}
    where, for every $N \in \mathbb{N}$, the truncated cone $\G_N(0)$ is defined by
    \begin{equation}\label{conoN}
        \G_N(0)=\{(y,t) \in \G(0) : 1/N < t < N\}.
    \end{equation}
    Note that the above limit exists because the integral is absolutely convergent. Indeed, for every $x \in (0,\infty)$,
    $S^q_\lambda(f)(x) \leq 2^{1/q} S_{\lambda,+}^q(f)(x).$ Then, according to Proposition~\ref{Lem_principal},
    since $\dim \E<\infty$, $S_\lambda^q$ is bounded from $L^2((0,\infty),\E)$ into $L^2(0,\infty)$. By applying Hölder's inequality and by
    taking into account that $f$ and $h$ have compact support we deduce that the integral under analysis is absolutely convergent.\\

    Interchanging the order of integration we get
    \begin{equation}\label{3.14}
        \langle t \partial_t P_t^\lambda(f)(x+y), h(x,y,t) \rangle
            = \lim_{N \to \infty} \int_0^\infty \langle f(z) , \Psi_N(h)(z) \rangle dz,
    \end{equation}
    where, for each $N \in \mathbb{N}$,
    \begin{equation*} %\label{PsiN}
        \Psi_N(h)(z) =  \int_{\G_N(0)} \int_0^\infty t \partial_t P_t^\lambda(x+y,z)h(x,y,t) dx \frac{dtdy}{t^2}, \quad z \in (0,\infty).
    \end{equation*}
    The interchange in the order of integration is justified because by using Hölder's inequality we obtain
    \begin{align}\label{finito}
         \int_0^\infty \int_0^\infty  \int_{\G_N(0)}  & |t \partial_t P_t^\lambda(x+y,z)|  \|h(x,y,t)\|_{\E^*}  \|f(z)\|_{\E}  \frac{dtdy}{t^2}  dx dz \nonumber \\
            \leq & C \|f\|_{L^\infty((0,\infty),\E)}  \int_{\supp(f)} \int_{\supp(h)}
                    \left(\int_{\G_N(0)} \|h(x,y,t)\|_{\E^*}^{q'} \frac{dtdy}{t^2} \right)^{1/q'} \nonumber \\
            & \times \left(\int_{\G_N(0)} |t \partial_t P_t^\lambda(x+y,z)|^{q} \frac{dtdy}{t^2} \right)^{1/q} dx dz \nonumber  \\
            \leq & C  \|f\|_{L^\infty((0,\infty),\E)} \|h\|_{L^\infty\left((0,\infty), L^{q'}\left(\G(0),\frac{dtdy}{t^2},\E^*\right) \right)} \nonumber \\
            & \times \int_{\supp(f)} \int_{\supp(h)}  \left(\int_{\G_N(x)} |t \partial_t P_t^\lambda(y,z)|^{q} \frac{dtdy}{t^2} \right)^{1/q} dx dz, \quad N \in \mathbb{N}.
    \end{align}
    Then, since $\supp(f)$ and $\supp(h)$ are compact, by using \eqref{Pt1} and \eqref{Pt2} we conclude that
    $$\int_0^\infty \int_0^\infty  \int_{\G_N(0)}   |t \partial_t P_t^\lambda(x+y,z)|  \|h(x,y,t)\|_{\E^*}  \|f(z)\|_{\E}   \frac{dtdy}{t^2}  dx dz
        < \infty, \quad N \in \mathbb{N}.$$

    For the incoming reasoning it is convenient to consider the operator
    $$g \longmapsto t \partial_t P_t^\lambda(\Psi_N(g))(x+y).$$
    In order to make this proof more legible the main properties of this operator will be shown in Section~\ref{sec:appendix} (Appendix 2).\\

    By interchanging the order of integration we can write, for each $x,t \in (0,\infty)$, $y \in \mathbb{R}$, and $N \in \mathbb{N}$,
    \begin{align}\label{(*b)}
        t \partial_t P_t^\lambda(\Psi_N(h))(x+y)
            = & \int_0^\infty  \int_{\G_N(0)} \int_0^\infty t \partial_t P_t^\lambda(x+y,z) s \partial_s P_s^\lambda(v+u,z)h(v,u,s) dv \frac{dsdu}{s^2} dz \nonumber\\
            = & \int_0^\infty \int_{\G_N(0)} k^\lambda_{s,t}(x,y;u,v)  h(v,u,s)\frac{dsdu}{s^2} dv
            = \Phi_N(h)(x,y,t),
    \end{align}
    where the kernel $k^{\lambda}_{s,t}$ is given by
    $$ k^\lambda_{s,t}(x,y;u,v) = \int_0^\infty t \partial_t P_t^\lambda(x+y,z) s \partial_s P_s^\lambda(v+u,z) dz,
        \quad v,x,s,t \in (0,\infty), \ u,y \in \mathbb{R}.$$
    The interchange in the order of integration is justified because the integral is absolutely convergent.
    Indeed, according to  \eqref{(*)}, \eqref{Pt1} and \eqref{Pt2} we have that
    \begin{align*}
        \int_0^\infty \int_0^\infty \int_{\G_N(0)}  & |t \partial_t P_t^\lambda(x+y,z)| \ |s \partial_s P_s^\lambda(v+u,z)| \ \|h(v,u,s)\|_{\E^*}  \frac{dsdu}{s^2} dv dz \\
            \leq & C \int_0^\infty \frac{t z^\lambda}{(||x+y|-z|+t)^{\lambda+2}} \int_{\supp(h)} \left(\int_{\G_N(0)} \|h(v,u,s)\|_{\E^*}^{q'} \frac{dsdu}{s^2} \right)^{1/q'} \\
            & \times \left(\int_{\G_N(0)} |s \partial_s P_s^\lambda(v+u,z)|^{q} \frac{dsdu}{s^2} \right)^{1/q} dv dz <\infty, \
            x,t \in (0,\infty), \ y \in \mathbb{R}, \ N \in \mathbb{N}.
    \end{align*}

    In Section~\ref{sec:appendix}, Proposition~\ref{Pro5.2}, we establish that the sequence
    of operators $\{\Phi_N\}_{N \in \mathbb{N}}$ is uniformly bounded from $H^1\left((0,\infty),L^{q'}\left(\G(0),\frac{dtdy}{t^2},\E^*\right)\right)$
    into $L^1\left((0,\infty),L^{q'}\left(\G(0),\frac{dtdy}{t^2},\E^*\right)\right)$.\\

    We now return back to \eqref{3.14}. Let $N \in \mathbb{N}$. We can write
    \begin{equation}\label{3.17}
    \int_0^\infty \langle f(z) ,\Psi_N(h)(z) \rangle dz
        = 4 \int_0^\infty \int_0^\infty \langle t\partial_tP_t^\lambda(f)(y) , t\partial_tP_t^\lambda(\Psi_N(h))(y) \rangle \frac{dydt}{t}.
    \end{equation}
    This equality can be shown by proceeding as in the proof of \cite[Proposition 4.4]{BCFR2} and by taking into account the following facts:
    \begin{itemize}
        \item $f \in L^\infty_c((0,\infty),\E)$.
        \item $(1+z^2)^{-1}\Psi_N(h) \in L^1((0,\infty),\E^*)$. Indeed, arguing as in \eqref{finito} it can be proved that
              $\Psi_N(h) \in L^\infty((0,\infty),\E^*)$.
        \item Since condition $(i)$ is assumed, if we define
              $$C_\lambda^q(f)(x)
                    =\left( \sup_{I \ni x} \frac{1}{|I|}\int_0^{|I|} \int_I  \|t \partial_t P_t^\lambda(f)(y)\|_\E^q \frac{dydt}{t} \right)^{1/q}, \quad x \in (0,\infty),$$
              where the supremum is taken over all bounded intervals $I \subset (0,\infty)$ such that $x \in I$, then $C_\lambda^q(f) \in L^\infty(0,\infty)$.
        \item $\Phi_N(h) \in L^1\left((0,\infty),L^{q'}\left(\G(0),\frac{dtdy}{t^2},\E^*\right)\right)$ because
              $h \in H^1\left((0,\infty),L^{q'}\left(\G(0),\frac{dtdy}{t^2},\E^*\right)\right)$ (Proposition~\ref{Pro5.2}).
    \end{itemize}
    By using Hölder's inequality and \eqref{3.17} if follows that (see \cite[Proposition 4.3]{BCFR2})
    \begin{align*}
        \left|\int_0^\infty \langle f(z) ,\Psi_N(h)(z) \rangle dz \right|
            \leq & C \int_0^\infty C_\lambda^q(f)(x) S_{\lambda,+}^{q'}(\Psi_N(h))(x)dx \\
            \leq & C \|C_\lambda^q(f)\|_{L^\infty(0,\infty)} \|S_{\lambda,+}^{q'}(\Psi_N(h))\|_{L^1(0,\infty)}.
    \end{align*}
    Finally, since $(i)$ holds we get
    \begin{align*}
        |\langle t \partial_t P_t^\lambda(f)(x+y), h(x,y,t) \rangle|
            \leq C & \|C_\lambda^q(f)\|_{L^\infty(0,\infty)} \limsup_{N \to \infty}  \|\Phi_N(h)\|_{L^1\left((0,\infty),L^{q'}\left(\G(0),\frac{dtdy}{t^2},\E^*\right)\right)} \\
            \leq C & \|f_o\|_{BMO_o(\mathbb{R},\E)} \|h\|_{H^1\left((0,\infty),L^{q'}\left(\G(0),\frac{dtdy}{t^2},\E^*\right)\right)},
    \end{align*}
    being $f_o$ the odd extension of $f$ to $\mathbb{R}$.\\

    Thus the proof of $(ii)$ is completed.
    \begin{flushright}
        \qed
    \end{flushright}

    %%%%%%%%%%%%%%%%%%%%%%%%%%%%%%%%%%%%%%%%%%%%%%%%%%%%%%%%%%%%%%%%%%%%%%%%%%%%%%%%%%%%%%%%%%%%%%%%%%%%%%%%%%%%%%%%%%%%
    \section{Proof of Theorem~\ref{Th_4.1_OX}} \label{sec:proof2}
    %%%%%%%%%%%%%%%%%%%%%%%%%%%%%%%%%%%%%%%%%%%%%%%%%%%%%%%%%%%%%%%%%%%%%%%%%%%%%%%%%%%%%%%%%%%%%%%%%%%%%%%%%%%%%%%%%%%%

    \subsection{Proof of $(ii) \Rightarrow (i)$} \label{subsec2:1->2}
        Assume that $(ii)$ holds. Let $f$ be an odd $\B$-valued function such that $\int_0^\infty \|f(z)\|_\B/(1+z^2) dz <\infty$.
        If
        $$\sup_{I} \frac{1}{|I|} \int_0^{|I|} \int_I  \|t \partial_t P_t^\lambda(f)(y)\|_\B^q \frac{dydt}{t}=\infty,$$
        where the supremum is taken over all the bounded intervals $I \subset (0,\infty)$, we have nothing to prove. Assume that
        $$\sup_{I} \frac{1}{|I|} \int_0^{|I|} \int_I  \|t \partial_t P_t^\lambda(f)(y)\|_\B^q \frac{dydt}{t}<\infty.$$
        According to \cite[Propositions 4.3 and 4.4]{BCFR2}, for every $g \in L^\infty_c(0,\infty) \otimes \B^*$,
        \begin{align}\label{18.1}
            \left| \int_0^\infty \langle f(x) , g(x) \rangle dx \right|
                = & 4 \left| \int_0^\infty \int_0^\infty  \langle t\partial_t P_t^\lambda(f)(y) , t\partial_t P_t^\lambda(g)(y) \rangle \frac{dy dt}{t} \right| \nonumber \\
                \leq &C  \int_0^\infty C_\lambda^q(f)(x) S_{\lambda,+}^{q'}(g)(x) dx
                \leq C \| C_\lambda^q(f) \|_{L^\infty(0,\infty)} \|S_{\lambda,+}^{q'}(g)\|_{L^1(0,\infty)},
        \end{align}
        being
        $$C_\lambda^q(f)(x)=\left( \sup_{I \ni x} \frac{1}{|I|} \int_I \int_0^{|I|} \|t \partial_t P_t^\lambda(f)(y)\|^q_{\B} \frac{dtdy}{t} \right)^{1/q}, \quad x \in (0,\infty),$$
        and
        $$ S_{\lambda,+}^{q'}(g)(x)
            = \left( \int_{\G_+(x)} \| t \partial_t P_t^\lambda(g)(y) \|^{q'}_{\B^*} \frac{dtdy}{t^2}\right)^{1/q'}, \quad x \in (0,\infty).$$

        Moreover, by using \cite[Colloraries 2.6 and 3.2]{Xu} and Proposition~\ref{Lem_principal} we deduce that
        \begin{equation} \label{18.2}
            \|S_{\lambda,+}^{q'}(g)\|_{L^1(0,\infty)}  \leq C  \|g\|_{H^1_o(\mathbb{R},\B^*)}, \quad g \in H^1_o(\mathbb{R},\B^*).
        \end{equation}

        Let $I = (a,b)$, being $0<a<b<\infty$. By applying \cite[Lemma 2.3]{GLY} and taking into account that $\overline{C_c(I) \otimes \B^*}=L^2(I,\B^*)$
        and $(f-f_I)_I=0$, we obtain,
        \begin{align*}
             \left(\frac{1}{|I|}\int_I \|f(x)-f_I\|_{\B}^2  dx \right)^{1/2}
                = & \frac{1}{|I|^{1/2}}
                    \sup_{\substack{g \in L^2(I,\B^*) \\ \|g\|_{L^2(I,\B^*)} \leq 1}}
                    \left| \int_I \langle f(x) - f_I , g(x) \rangle dx \right| \\
                = & \sup_{\substack{g \in C_c(I) \otimes \B^* \\ \|g\|_{L^2(I,\B^*)} \leq 1}}
                    \left| \int_I \langle f(x), \frac{g(x) - g_I}{|I|^{1/2}} \rangle dx \right|.
        \end{align*}
        If $g \in C_c(I)\otimes \B^*$ it is clear that $h=(g-g_I)\chi_I/2|I|^{1/2}$ is a $2$-atom in $(0,\infty)$ for $H^1_o(\mathbb{R},\B^*)$, because       $\supp(h) \subset I$, $h_I=0$ and
        \begin{align*}
            \|h\|_{L^2((0,\infty),\B^*)}
                \leq & \frac{\sqrt{2}}{2|I|^{1/2}} \left( \int_I \|g(x)\|^2_{\B^*}dx + \frac{1}{|I|}  \left(\int_I \|g(x)\|_{\B^*}dx\right)^2  \right)^{1/2}
                \leq \frac{1}{|I|^{1/2}}.
        \end{align*}
        Hence, by applying \eqref{18.1} and \eqref{18.2}, we deduce
        \begin{align*}
            \left(\frac{1}{|I|}\int_I \|f(x)-f_I\|_{\B}^2  dx \right)^{1/2}
                \leq & C \| C_\lambda^q(f) \|_{L^\infty(0,\infty)} \sup_{\substack{g \in C_c(I) \otimes \B^* \\ \|g\|_{L^2(I,\B^*)} \leq 1}}
                       \left\| \frac{g - g_I}{2|I|^{1/2}} \right\|_{H^1_o(\mathbb{R},\B^*)}
                \leq  C \| C_\lambda^q(f) \|_{L^\infty(0,\infty)}.
        \end{align*}

        Suppose now that $I=(0,\beta)$, for some $\beta>0$. By proceeding as before we get
        \begin{align*}
             \left(\frac{1}{\beta}\int_0^\beta \|f(x)\|_{\B}^2  dx \right)^{1/2}
                = & \sup_{\substack{g \in C_c(0,\beta) \otimes \B^* \\ \|g\|_{L^2((0,\beta),\B^*)} \leq 1}}
                    \left| \int_0^\beta \langle f(x), \frac{g(x)}{\beta^{1/2}} \rangle dx \right|
        \end{align*}
        and the same conclusion follows because $h=g\chi_{(0,\beta)}/\beta^{1/2}$ satisfies $\supp(h) \subset [0,\beta]$ and $\|h\|_{L^2((0,\infty),\B^*)} \leq 1/\beta^{1/2}$.\\

        Thus $(i)$ is established.
        \begin{flushright}
        \qed
        \end{flushright}

    \subsection{Proof of $(i) \Rightarrow (ii)$} \label{subsec2:1->2}

        Suppose that $(i)$ holds. In order to see $(ii)$, according to Proposition~\ref{Lem_principal},
        \cite[Colloraries 2.6 and 3.2]{Xu} and \eqref{equivalentes}, we prove that, for some $C>0$,
        \begin{equation}\label{19.1}
            \| S_{\lambda}^{q'} (g) \|_{L^{q'}(0,\infty)} \leq C \|g\|_{L^{q'}((0,\infty),\B^*)}, \quad g \in L^{q'}((0,\infty),\B^*).
        \end{equation}

        Moreover, it is sufficient to see that, there exists $C>0$ such that
        \begin{equation}\label{19.2}
            \| S_\lambda^{q'} (g) \|_{L^{q'}(0,\infty)} \leq C \|g\|_{L^{q'}((0,\infty),\E^*)}, \quad g \in L^{q'}((0,\infty),\E^*),
        \end{equation}
        for every subspace $\E$ of $\B$, being $\dim \E < \infty$.
        Indeed, assume that \eqref{19.2} holds and take $g \in L^{q'}((0,\infty),\B^*)$. By \cite[Lemma 2.3]{GLY}, we can write
        \begin{align*}
            \| S_{\lambda}^{q'} (g) \|_{L^{q'}(0,\infty)}
                & = \|t\partial_t P_t^\lambda(g)(x+y)\|_{L^{q'}\left((0,\infty) \times \G(0),\frac{dtdy}{t^2}dx,\B^*\right) } \\
                & = \sup_{\substack{G \in L^{q}\left((0,\infty) \times \G(0),\frac{dtdy}{t^2}dx,\B\right)  \\
                                    \|G\|_{L^{q}\left((0,\infty) \times \G(0),\frac{dtdy}{t^2}dx,\B\right) }\leq 1}}
                    \left|\int_0^\infty \int_{\G(0)} \langle t \partial_t P_t^\lambda(g)(x+y),  G(x,y,t) \rangle \frac{dtdy}{t^2} dx \right| \\
                & = \sup_{\substack{G \in L^{q}\left((0,\infty)\times\G(0),\frac{dtdy}{t^2}dx  \right)\otimes \B \\
                                    \|G\|_{L^{q}\left((0,\infty) \times \G(0),\frac{dtdy}{t^2}dx,\B\right) }\leq 1}}
                    \left|\int_0^\infty \int_{\G(0)} \langle t \partial_t P_t^\lambda(g)(x+y),  G(x,y,t) \rangle \frac{dtdy}{t^2} dx \right|.
        \end{align*}
        Observe that in the last equality we have applied that $L^{q}\left((0,\infty)\times\G(0),\frac{dtdy}{t^2}dx  \right)\otimes \B$
        is a dense subspace of $L^{q}\left((0,\infty) \times \G(0),\frac{dtdy}{t^2}dx,\B\right)$. Fix $\varepsilon>0$. There exists
        $G \in L^{q}\left((0,\infty) \times \G(0),\frac{dtdy}{t^2}dx  \right)\otimes \B$, such that
        $\|G\|_{L^{q}\left((0,\infty) \times \G(0),\frac{dtdy}{t^2}dx,\B\right) }\leq 1$ and
        \begin{align*}
            \| S_{\lambda}^{q'} (g) \|_{L^{q'}(0,\infty)}
                \leq \left|\int_0^\infty \int_{\G(0)} \langle t \partial_t P_t^\lambda(g)(x+y),  G(x,y,t) \rangle \frac{dtdy}{t^2} dx \right|  + \varepsilon,
        \end{align*}
        being $G = \sum_{j=1}^n a_j G_j$, $a_j \in \B$, $G_j \in L^{q}\left((0,\infty) \times \G(0),\frac{dtdy}{t^2}dx  \right) $, $j=1, \dots, n$
        and $n \in \mathbb{N}$. If we define $\E=\spann\{a_j\}_{j=1}^n$, it is clear that
        \begin{align*}
            \langle t \partial_t P_t^\lambda(g)(x+y),  G(x,y,t) \rangle_{\B^* \times \B}
                = & \langle t \partial_t P_t^\lambda(g)(x+y),  G(x,y,t) \rangle_{\E^* \times \E}, \quad x\in (0,\infty), \ (y,t) \in \G(0),
        \end{align*}
        because every element of  $\B^*$ can be seen by restriction as an element of $\E^*$. Hence, by Hölder's inequality and \eqref{19.2} we conclude that
        \begin{align*}
            \| S_{\lambda}^{q'} (g) \|_{L^{q'}(0,\infty)}
                & \leq \left|\int_0^\infty \int_{\G(0)} \langle t \partial_t P_t^\lambda(g)(x+y),  G(x,y,t) \rangle_{\E^* \times \E} \frac{dtdy}{t^2} dx \right|  + \varepsilon \\
                & \leq \|t \partial_t P_t^\lambda(g)(x+y)\|_{L^{q'}\left((0,\infty) \times \G(0),\frac{dtdy}{t^2}dx,\E^*\right) }
                       \|G\|_{L^{q}\left((0,\infty) \times \G(0),\frac{dtdy}{t^2}dx,\E\right) } + \varepsilon \\
                & \leq C\|g\|_{L^{q'}\left((0,\infty), \E^*\right)} + \varepsilon
                \leq C\|g\|_{L^{q'}\left((0,\infty), \B^*\right)} + \varepsilon,
        \end{align*}
        and this gives \eqref{19.1}.\\

        Let $\E$ be a finite dimensional subspace of $\B$. In order to prove \eqref{19.2}, by the equivalences shown in Proposition~\ref{Lem_principal},
        we are going to see that there exists $C>0$, independent of $\E$, such that
        \begin{equation*}\label{4.3XU}
            \| S_\lambda^{q'} (g) \|_{L^1(0,\infty)} \leq C \|g\|_{H^1_o(\mathbb{R},\E^*)}, \quad g \in H^{1}_o(\mathbb{R},\E^*).
        \end{equation*}
        Fix $g \in H^{1}_o(\mathbb{R},\E^*)$.
        We denote, for every $N \in \mathbb{N}$, $\E_N=L^{q}\left( \G_N(0),\dfrac{dtdy}{t^2}, \E \right)$, where the truncated cone $\G_N(0)$
        is defined in \eqref{conoN}. By invoking \cite[Corollary III.2.13]{DU}  we have that
        $\E_N^*=L^{q'}\left( \G(0),\dfrac{dtdy}{t^2}, \E^* \right)$ and
        $\left(L^1((0,\infty),\E_N^*)\right)^* = L^\infty((0,\infty),\E_N)$.\\

        It is clear that
        $$S_\lambda^{q'}(g)(x)
            = \lim_{N \to \infty} \left( \int_{\G_N(0)} \|t \partial_t P_t^\lambda(g)(x+y) \|_{\E^*}^{q'} \frac{dtdy}{t^2} \right)^{1/q'}, \quad x \in (0,\infty).$$

        Let $m \in \mathbb{N}$ and $N \in \mathbb{N}$. Assume that $G \in L^\infty_c([0,\infty),\E^*)$. Estimations \eqref{(*)}, \eqref{Pt1} and \eqref{Pt2} lead to
        \begin{align*}
            \int_0^m &  \left( \int_{\G_N(0)} \| t \partial_t P_t^\lambda (G)(x+y)\|_{\E^*}^{q'} \frac{dtdy}{t^2} \right)^{1/q'} dx \\
                \leq & C \int_0^m  \left( \int_{-N}^{N} \int_{1/N}^N
                    \left( \int_{\supp(G)} \frac{tz^\lambda}{(||x+y|-z|+t)^{\lambda+2}} \|G(z)\|_{E^*} dz \right)^{q'}  \frac{dtdy}{t^2} \right)^{1/q'} dx \\
                \leq & C \int_0^m  \left( \int_{-N}^{N} \int_{1/N}^N
                    \left( \left(\int_{\supp(G) \cap [0,2(m+N)]} + \int_{\supp(G)\cap [2(m+N),\infty)} \right) \right. \right.\\
                    & \left. \left. \times \frac{tz^\lambda}{(||x+y|-z|+t)^{\lambda+2}} \|G(z)\|_{E^*} dz \right)^{q'}  \frac{dtdy}{t^2} \right)^{1/q'} dx \\
                \leq & C \|G\|_{L^\infty((0,\infty),\E^*)} \int_0^m  \left( \int_{-N}^{N} \int_{1/N}^N
                      \left( \left( \frac{(2(m+N))^\lambda}{t^{\lambda+1}}  + \frac{1}{t} \right) \int_{\supp(G)} dz \right)^{q'}  \frac{dtdy}{t^2} \right)^{1/q'} dx\\
                \leq & C  \|G\|_{L^\infty((0,\infty),\E^*)} |\supp(G)|,
        \end{align*}
        where $C>0$ does not depend on $G$.\\

        Hence, recalling the atomic representation of the elements of $H^1_o(\mathbb{R},\E^*)$ we deduce that
        $t\partial_t P_t^\lambda(g) \in L^1\left((0,m),L^{q'}\left(\G_N(0),\dfrac{dtdy}{t^2},\E^*\right)\right)$.\\

        According to \cite[Lemma 2.3]{GLY} we have that
        \begin{align*}
            & \left\|  \left( \int_{\G_N(0)} \| t \partial_t P_t^\lambda (g)(x+y)\|_{\E^*}^{q'} \frac{dtdy}{t^2} \right)^{1/q'} \right\|_{L^1(0,m)} \\
            & \qquad \qquad = \sup_{\substack{ h \in L^\infty((0,m),\E_N) \\ \|h\|_{L^\infty((0,m),\E_N)} \leq 1}}
               \left| \int_0^m \int_{\G_N(0)} \langle t \partial_t P_t^\lambda (g)(x+y) , h(x,y,t) \rangle \frac{dtdy}{t^2} dx \right|.
        \end{align*}
        Let $h \in L^\infty((0,m),\E_N)$ such that $\|h\|_{L^\infty((0,m),\E_N)} \leq 1$. By using Hölder's inequality and repeating the above
        manipulations we can see that
        $$\int_0^m \int_{\G_N(0)} \int_0^\infty  \left| \langle t \partial_t P_t^\lambda (x+y,z)g(z) , h(x,y,t) \rangle \right| dz \frac{dtdy}{t^2} dx
            < \infty.$$
        Then, we can interchange the order of integration to write
        \begin{align*}
           &   \int_0^m \int_{\G_N(0)} \langle t \partial_t P_t^\lambda(g)(x+y) , h(x,y,t) \rangle  \frac{dtdy}{t^2} dx
            = \int_0^\infty \langle g(z) , \Psi_N( \chi_{(0,m)}(x) h)(z) \rangle  dz,
        \end{align*}
        where, as above,
        $$\Psi_N(H)(z)
            = \int_{\G_N(0)} \int_0^\infty t \partial_t P_t^\lambda(x+y,z) H(x,y,t) dx \frac{dtdy}{t^2}, \quad z \in (0,\infty).$$

        According to \eqref{(*)} and \eqref{3.2a} we get
        \begin{align*}
            \| \Psi_N( \chi_{(0,m)}(x) h)(z) \|_{\E_N}
                & \leq \int_0^m \int_{\G_N(0)}  |t \partial_t P_t^\lambda(x+y,z)| \ \|h(x,y,t)\|_\E   \frac{dtdy}{t^2} dx\\
                & \leq C \int_0^m \int_{\G_N(0)}  \frac{t}{(|x+y|-z)^2+t^2} \ \|h(x,y,t)\|_\E   \frac{dtdy}{t^2} dx\\
                & \leq C \int_0^m \left( \int_{-N}^{N} \int_{1/N}^N  \left(\frac{t}{(|x+y|-z)^2+t^2}\right)^{q'}  \frac{dtdy}{t^2} \right)^{1/q'} dx
                  \leq C, \ z \in (0,\infty),
        \end{align*}
        where $C>0$ depends on $m$ and $N$. Thus, this function is locally integrable.\\

        Suppose for a moment that there exists $C>0$ independent of $\E$, $m$ and $N \in \mathbb{N}$ such that
        \begin{equation}\label{objetivo6}
            \sup_{I} \frac{1}{|I|} \int_0^{|I|} \int_I \|t \partial_t P_t^\lambda(\Psi_N(\chi_{(0,m)}h))(x)\|_\E^q \frac{dxdt}{t}
                \leq C,
        \end{equation}
        where the supremum is taken over all the bounded intervals $I \subset (0,\infty)$. Since $(i)$ holds, by using duality,
        \eqref{objetivo6} leads to
        \begin{align*}
            \left| \int_0^\infty \langle g(z) , \Psi_N( \chi_{(0,m)}(x) h)(z) \rangle  dz \right|
                & \leq C \|g\|_{H^1_o(\mathbb{R},\E^*)} \| \left(\Psi_N(\chi_{(0,m)}(x) h)\right)_o\|_{BMO_o(\mathbb{R},\E)} \\
                & \leq C \|g\|_{H^1_o(\mathbb{R},\E^*)},
        \end{align*}
        where $C$ depends neither on $m,N \in \mathbb{N}$ nor on $\E$.\\

        We conclude, by taking $m \to \infty$, that
        $$\left\|  \left( \int_{\G_N(0)} \| t \partial_t P_t^\lambda (g)(x+y)\|_{\E^*}^{q'} \frac{dtdy}{t^2} \right)^{1/q'} \right\|_{L^1(0,\infty)}
            \leq C \|g\|_{H^1_o(\mathbb{R},\E^*)},$$
        where $C>0$ is independent of $N \in \mathbb{N}$, and by taking $N \to \infty$, it follows that
        \begin{align*}
            \|S_\lambda^{q'}(g)\|_{L^1(0,\infty)}
                \leq & C \|g\|_{H^1_o(\mathbb{R},\E^*)}.
        \end{align*}

        We now prove \eqref{objetivo6}. Fix $m,N \in \mathbb{N}$ and a bounded interval $I \subset (0,\infty)$. We decompose
        $H_{m}=\chi_{(0,m)}h$ as follows
        $$H_{m}
            = H_{m} \chi_{2I} + H_{m} \chi_{(0,\infty) \setminus 2I}
            = H_{m}^1 + H_{m}^2.$$
        By proceeding as in \eqref{3.1.a} and by using Proposition~\ref{Pro5.2} we get
        \begin{align}\label{4.4}
            \frac{1}{|I|} \int_0^{|I|} & \int_I  \|t \partial_t P_t^\lambda(\Psi_N(H_{m}^1))(x)\|_\E^q \frac{dxdt}{t}
                =  \frac{1}{|I|} \int_0^{|I|} \int_I \|\Phi_N(H_{m}^1)(x,0,t)\|_\E^q \frac{dxdt}{t} \nonumber \\
                \leq & C  \frac{1}{|I|}   \|\Phi_N(H_{m}^1)\|^q_{L^q\left((0,\infty),L^q\left( \G(0),\frac{dtdy}{t^2},\E \right)\right)}
                \leq  C   \frac{1}{|I|}  \int_{2I} \|H_{m}(z,x,t)\|^q_{\E_N} dz
                \leq C  \|h\|^q_{L^\infty((0,m),\E_N)}
                \leq C,
        \end{align}
        where $C>0$ does not depend on $m$ and $N$. Here $\Phi_N$ is defined as in \eqref{(*b)}.\\

        On the other hand, for each $x,t \in (0,\infty)$,
        \begin{align*}
            \|t \partial_t P_t^\lambda(\Psi_N(H_{m}^2))(x)\|_\E
                \leq & \int_{(0,m) \setminus 2I} \int_{\G_N(0)} |k_{s,t}^\lambda(x,0;u,v)| \|h(v,u,s)\|_\E  \frac{dsdu}{s^2} dv.
        \end{align*}

        We claim that
        \begin{equation}\label{kst}
            |k^\lambda_{s,t}(x,y;u,v)|
                \leq C \frac{st}{(||x+y|-|u+v||+s+t)^3}, \quad v,x,s,t \in (0,\infty), \ u,y \in \mathbb{R}.
        \end{equation}
        Indeed, let $v,x,s,t \in (0,\infty)$ and $u,y \in \mathbb{R}$.
        Since $\{P_t^\lambda\}_{t \geq 0}$ is a semigroup of operators we have that
        $$\int_0^\infty P_t^\lambda(x,z)P_s^\lambda(z,y)dz=P_{t+s}^\lambda(x,y).$$
        Then, we can write
        \begin{align*} %\label{partial2}
            k^\lambda_{s,t}(x,y;u,v)
                = & t s \partial_t \partial_s \int_0^\infty  P_t^\lambda(x+y,z)  P_s^\lambda(z,u+v) dz
                = t s \partial_t \partial_s P_{t+s}^\lambda(x+y,u+v) \nonumber \\
                = & t s \partial_r^2 P_r^\lambda(x+y,u+v)_{|{r=t+s}}.
        \end{align*}
        By \cite[in the bottom of the p. 280]{BFMT} it follows that \eqref{kst} holds.\\

        From \eqref{kst} we deduce that, for every $x,v,t \in (0,\infty)$,
        \begin{align*}
            \int_{\G_N(0)}& |k_{s,t}^\lambda(x,0;u,v)| \|h(v,u,s)\|_\E  \frac{dsdu}{s^2}
                \leq  \|h\|_{L^\infty((0,m),\E_N)} \left( \int_{\G_N(0)} |k_{s,t}^\lambda(x,0;u,v)|^{q'} \frac{dsdu}{s^2} \right)^{1/q'} \\
                \leq & C  t  \left( \int_{\mathbb{R}} \int_{|u|}^\infty \frac{dsdu}{(|x-|u+v||+s+t)^{2q'+2}}  \right)^{1/q'}
                \leq  C  t  \left( \int_{\mathbb{R}} \frac{du}{(|x-|u+v||+|u|+t)^{2q'+1}}  \right)^{1/q'}.
        \end{align*}
        In order to estimate the last integral we distinguish several cases. We have that
        \begin{align*} %\label{I}
            \int_0^\infty  \frac{du}{(|x-u-v|+u+t)^{2q'+1}}
                \leq & \int_0^{|x-v|} \frac{du}{(|x-v|+t)^{2q'+1}}  + \int_{|x-v|}^\infty \frac{du}{(-|x-v|+2u+t)^{2q'+1}} \nonumber \\
                 \leq &  \frac{C}{(|x-v|+t)^{2q'}}, \quad x,v,t \in (0,\infty),
        \end{align*}
        and
        \begin{align*} %\label{II}
            \int_0^\infty & \frac{du}{(|x-|u-v||+u+t)^{2q'+1}}
                \leq  \int_0^{\max\{\min\{v,x-v\},0\}} \frac{du}{(x-v+t)^{2q'+1}}  \nonumber \\
                & + \int_{\max\{\min\{v,x-v\},0\}}^v \frac{du}{(2u+v-x+t)^{2q'+1}}
                 +    \int_v^{x+v} \frac{du}{(x+v+t)^{2q'+1}}
                     + \int_{x+v}^\infty \frac{du}{(2u-x-v+t)^{2q'+1}}  \nonumber \\
               \leq &   \frac{C}{(|x-v|+t)^{2q'}}, \quad x,v,t \in (0,\infty).
        \end{align*}
        Hence, we obtain
        $$ \int_{\G_N(0)} |k_{s,t}^\lambda(x,0;u,v)| \|h(v,u,s)\|_\E  \frac{dsdu}{s^2}
            \leq C \frac{t}{(|x-v|+t)^{2}}, \quad x,v,t \in (0,\infty).$$
        Finally, we deduce
        \begin{align}\label{45.1}
            \frac{1}{|I|} \int_0^{|I|} \int_I \|t \partial_t P_t^\lambda(\Psi_N(H_{m}^2))(x)\|_\E^q \frac{dxdt}{t}
                \leq C & \frac{1}{|I|} \int_0^{|I|} \int_I t^{q-1} \left(\int_{(0,\infty) \setminus 2I} \frac{dv}{|x-v|^2} \right)^q dxdt
                \leq C,
        \end{align}
        where $C>0$ depends neither on $m,N$  nor on $I$.\\

        By combining \eqref{4.4} and \eqref{45.1} we establish \eqref{objetivo6}.\\

        Thus the proof of this theorem is completed.
    \begin{flushright}
        \qed
    \end{flushright}

%\newpage
    %%%%%%%%%%%%%%%%%%%%%%%%%%%%%%%%%%%%%%%%%%%%%%%%%%%%%%%%%%%%%%%%%%%%%%%%%%%%%%%%%%%%%%%%%%%%%%%%%%%%%%%%%%%%%%%%%%%%
    \section{Appendices} \label{sec:appendix}
    %%%%%%%%%%%%%%%%%%%%%%%%%%%%%%%%%%%%%%%%%%%%%%%%%%%%%%%%%%%%%%%%%%%%%%%%%%%%%%%%%%%%%%%%%%%%%%%%%%%%%%%%%%%%%%%%%%%%

    In this section we show two results that have been very useful in the proof of Theorems~\ref{Th_3.1_OX} and \ref{Th_4.1_OX}.

    \subsection{Appendix 1}\label{subsec:app1}

    The following property was used in Subsection~\ref{subsec:1->2}.

    \begin{Prop}\label{Prop_BMO_H1}
        Let $\B$ be a Banach space, $\lambda>0$ and $2 \leq q < \infty$. Suppose that
        \begin{equation}\label{BMO-L^inf}
            \|S^q_\lambda(f)\|_{BMO(0,\infty)} \leq C \|f\|_{L^\infty((0,\infty),\B)}, \quad f \in L^\infty_{c}((0,\infty),\B).
        \end{equation}
        Then,
        \begin{equation}\label{H1-L1}
            \|S^q_{\lambda,+}(f)\|_{L^1(0,\infty)} \leq C \|f\|_{H^1_o(\mathbb{R},\B)}, \quad f \in H^1_o(\mathbb{R},\B).
        \end{equation}
    \end{Prop}

    \begin{proof}
        According to \eqref{equivalentes} to show \eqref{H1-L1} it is sufficient to see that, for every $f \in H^1_o(\mathbb{R},\B)$,
        \begin{equation}\label{5.3}
            \|S^q_{\lambda}(f)\|_{L^1(0,\infty)} \leq C \|f\|_{H^1_o(\mathbb{R},\B)}.
        \end{equation}
        Let $f \in H^1_o(\mathbb{R},\B)$. We write $f=\sum_{j=1}^\infty \lambda_j a_j$ on $(0,\infty)$ where, for every $j \in \mathbb{N}$,
        $a_j$ is an $o$-atom and $\lambda_j \in \mathbb{C}$, being $\sum_{j=1}^\infty |\lambda_j|<\infty$. Here by $o$-atom we mean the class
        of atoms defined in the introduction, as follows: $a$ is an
        $\infty$-atom supported on $(0,\infty)$ or $a=b \chi_{(0,\delta)}/\delta$, for a certain $b \in \B$, being $\|b\|_\B=1$
        and $\delta>0$. \\

        We have that
        $$\partial_t \int_0^\infty  P_t^\lambda(y,z) f(z)dz
            =  \sum_{j=1}^\infty \lambda_j \int_0^\infty \partial_t P_t^\lambda(y,z) a_j(z)dz , \quad t\in (0,\infty), \ y \in \mathbb{R}.$$
        This equality is justified because the serie
        $$\sum_{j=1}^\infty |\lambda_j| \int_0^\infty |\partial_t P_t^\lambda(y,z)| \|a_j(z)\|_\B dz$$
        is uniformly convergent in $y \in \mathbb{R}$ ant $t \in K$, for every compact subset $K \subset (0,\infty)$. Indeed, let $K$ be a compact subset of
        $(0,\infty)$. By \eqref{(*)} and \eqref{3.2a} we get
        $$|\partial_t P_t^\lambda(y,z)|
            \leq C, \quad t \in K, \ y \in \mathbb{R}, \ z \in (0,\infty).$$
        Then, since $a_j$ is an $o$-atom, for every $j \in \mathbb{N}$, it follows that
        \begin{align*}
            \sum_{j=1}^\infty |\lambda_j| \int_0^\infty |\partial_t P_t^\lambda(y,z)| \|a_j(z)\|_\B dz
                \leq C \sum_{j=1}^\infty |\lambda_j|
                < \infty, \quad t \in K, \ y \in \mathbb{R}.
        \end{align*}
        Hence, we can write
        \begin{align*}
            S_\lambda^q(f)(x)
                =  & \left( \int_{\G(x)} \left\| t \partial_t P_t^\lambda\left( \sum_{j=1}^\infty \lambda_j a_j \right) (y) \right\|^q_\B  \frac{dtdy}{t^2}\right)^{1/q} \nonumber \\
                =  & \left( \int_{\G(x)} \left\|\sum_{j=1}^\infty \lambda_j t \partial_t P_t^\lambda\left(  a_j \right) (y) \right\|^q_\B  \frac{dtdy}{t^2}\right)^{1/q}
                \leq   \sum_{j=1}^\infty |\lambda_j| S_\lambda^q(a_j)(x), \quad x \in (0,\infty).
        \end{align*}
        In order to see \eqref{5.3} it is sufficient to show that there exists $C>0$ such that
        \begin{equation}\label{L1-atom}
            \|S_\lambda^q(a)\|_{L^1(0,\infty)} \leq C,
        \end{equation}
        for every $o$-atom $a$. \\

        To prove \eqref{L1-atom} we use a procedure employed by Journé (\cite[p.  49-51]{J}).
        Let $a$ be an $o$-atom supported in the interval $I=(z_0-|I|/2,z_0+|I|/2) \subset (0,\infty)$ and
        $\|a\|_{L^\infty((0,\infty),\B)} \leq 1/|I|$. We denote again $2I=(z_0-|I|,z_0+|I|) \cap (0,\infty)$ and
        by $J$ we represent the interval $(z_0+|I|,z_0+|I|+|2I|)$.\\

        By \cite[Lemma 1.1 (b), p. 217]{G} and \eqref{BMO-L^inf} we get
        \begin{align*}
            \frac{1}{|2I|} \int_{2I} S_\lambda^q(a)(x)dx
                \leq & \frac{1}{|2I|} \int_{2I} |S_\lambda^q(a)(x) - S_\lambda^q(a)_{2I} | dx + |S_\lambda^q(a)_{2I} - S_\lambda^q(a)_{J}| + S_\lambda^q(a)_{J} \\
                \leq & C \|S_\lambda^q(a)\|_{BMO(0,\infty)} + S_\lambda^q(a)_{J}
                \leq C \|a\|_{L^\infty((0,\infty),\B)} + S_\lambda^q(a)_{J}
                \leq \frac{C}{|I|} + S_\lambda^q(a)_{J}.
        \end{align*}
        Hence, since $J \subset (0,\infty) \setminus 2I$ and $|J|=|2I|$,
        \begin{align*}
            \|S_\lambda^q(a)\|_{L^1(0,\infty)}
                = & \int_{2I} S_\lambda^q(a)(x)dx + \int_{(0,\infty) \setminus 2I} S_\lambda^q(a)(x)dx
                \leq C + |2I| S_\lambda^q(a)_{J} + \int_{(0,\infty) \setminus 2I} S_\lambda^q(a)(x)dx \\
                \leq & C \left(1 + \int_{(0,\infty) \setminus 2I} S_\lambda^q(a)(x)dx \right).
        \end{align*}
        By writing $P_t^\lambda(y,z)=P_{t,1}^\lambda(y,z) + P_{t,2}^\lambda(y,z)$, $t,z \in (0,\infty)$, $y \in \mathbb{R}$, where
        $$P_{t,1}^\lambda(y,z)
                = \frac{2\lambda (|y|z)^\lambda t}{\pi} \int_0^{\pi/2} \frac{(\sin \theta)^{2\lambda-1}}{[(|y|-z)^2+t^2+2|y|z(1-\cos \theta)]^{\lambda+1}}d\theta, $$
        it follows that
        \begin{align}\label{objetivo5}
            \int_{(0,\infty) \setminus 2I} S_\lambda^q(a)(x)dx
                \leq & \int_{(0,\infty) \setminus 2I}  \left\| \int_0^\infty t \partial_t P_{t,1}^\lambda(y,z)a(z)dz \right\|_{L^q\left(\G(x),\frac{dtdy}{t^2},\B\right)} dx \nonumber \\
                & + \int_{(0,\infty) \setminus 2I}  \left\| \int_0^\infty t \partial_t P_{t,2}^\lambda(y,z)a(z)dz \right\|_{L^q\left(\G(x),\frac{dtdy}{t^2},\B\right)} dx .
        \end{align}
        According to \eqref{(*)} and \eqref{||Pt2||}  and using Minkowski's inequality we have that
        \begin{align*}
            \int_{(0,\infty) \setminus 2I} & \left\| \int_0^\infty t \partial_t P_{t,2}^\lambda(y,z)a(z)dz \right\|_{L^q\left(\G(x),\frac{dtdy}{t^2},\B\right)} dx \\
                \leq & 2 \int_{(0,\infty) \setminus 2I} \int_0^\infty \left\|  t \partial_t P_{t,2}^\lambda(y,z) \right\|_{L^q\left(\G_+(x),\frac{dtdy}{t^2}\right)} \|a(z)\|_\B dz dx\\
                \leq & C \int_{(0,\infty) \setminus 2I} \int_I \frac{z^\lambda}{(x+z)^{\lambda+1}} \|a(z)\|_\B dz dx
                \leq  \frac{C}{|I|} \int_I z^\lambda \int_{(0,\infty) \setminus 2I}  \frac{1}{(x+z)^{\lambda+1}}  dx dz \\
                \leq & \frac{C}{|I|} \int_I z^\lambda \left( \frac{1}{(z_0+|I|+z)^{\lambda}}  + c(I) \left(\frac{1}{z^\lambda} + \frac{1}{(z_0-|I|+z)^{\lambda}} \right)\right) dz
                \leq C,
        \end{align*}
        where $C>0$ does not depend on $a$. Here $c(I)=0$, when $z_0 \leq |I|$, and $c(I)=1$, provided that $z_0>|I|$.\\

        We now deal with the integral involving $P_{t,1}^\lambda$ in \eqref{objetivo5}.
        Assume first that $a=b\chi_{(0,\delta)}/\delta$, where $\delta>0$ and $b \in \B$ such that $\|b\|_\B=1$.
        By using Minkowski's inequality and \eqref{||Pt1||} it follows that
        \begin{align*}
            \int_{(0,\infty) \setminus 2I} & \left\| \int_0^\infty t \partial_t P_{t,1}^\lambda(y,z)a(z)dz \right\|_{L^q\left(\G(x),\frac{dtdy}{t^2},\B\right)} dx
                \leq \frac{\|b\|_\B}{\delta}\int_{2\delta}^\infty  \left\| \int_0^\delta t \partial_t P_{t,1}^\lambda(y,z)dz \right\|_{L^q\left(\G(x),\frac{dtdy}{t^2}\right)} dx \\
                \leq & \frac{2}{\delta}\int_{2\delta}^\infty \int_0^\delta \left\|  t \partial_t P_{t,1}^\lambda(y,z) \right\|_{L^q\left(\G_+(x),\frac{dtdy}{t^2}\right)} dz dx
                \leq  \frac{C}{\delta} \int_0^\delta \delta^\lambda \int_{2\delta}^\infty \frac{dx}{|x-z|^{\lambda+1}} dz
                \leq C,
        \end{align*}
        where $C>0$ is independent of $a$.\\

        Suppose now that $\int_I a(z)dz=0$. By the fundamental theorem of calculus and Minkowski's inequality we can write
        \begin{align*}
            \int_{(0,\infty) \setminus 2I}  & \left\| \int_I t\partial_t P_{t,1}^\lambda(y,z)  a(z)dz \right\|_{L^q\left(\G(x),\frac{dtdy}{t^2},\B\right)} dx \\
                = & \int_{(0,\infty) \setminus 2I}   \left\| \int_I \left[t\partial_t P_{t,1}^\lambda(y,z) - t\partial_t P_{t,1}^\lambda(y,z_0)\right] a(z)dz \right\|_{L^q\left(\G(x),\frac{dtdy}{t^2},\B\right)} dx\\
                \leq  & \int_{(0,\infty) \setminus 2I}   \int_I \|a(z)\|_B  \|t\partial_t P_{t,1}^\lambda(y,z) - t\partial_t P_{t,1}^\lambda(y,z_0)\|_{L^q\left(\G(x),\frac{dtdy}{t^2}\right)}  dz dx \\
                \leq  & \frac{1}{|I|} \int_{(0,\infty) \setminus 2I} \int_I \left| \int_{z_0}^z \|t\partial_t \partial_u P_{t,1}^\lambda(y,u)\|_{L^q\left(\G(x),\frac{dtdy}{t^2}\right)} du \right| dz dx.
        \end{align*}
        We are going to see that
        \begin{equation}\label{5.6}
            \frac{1}{|I|} \int_{(0,\infty) \setminus 2I} \int_I \left| \int_{z_0}^z \|t\partial_t \partial_u P_{t,1}^\lambda(y,u)\|_{L^q\left(\G(x),\frac{dtdy}{t^2}\right)} du \right| dz dx
                \leq C,
        \end{equation}
        where $C>0$ does not depend on $I$.\\

        We have that, for every $u,t \in (0,\infty)$ and $y \in \mathbb{R}$,
        \begin{align*}
            \partial_t \partial_u P_{t,1}^\lambda(y,u)
                = & \frac{2\lambda^2 |y|^\lambda u^{\lambda-1}}{\pi} \int_0^{\pi/2} \frac{(\sin \theta)^{2\lambda-1}}{[(|y|-u)^2+t^2+2|y|u(1-\cos \theta)]^{\lambda+1}} d\theta \\
                & - \frac{4\lambda^2(\lambda+1) |y|^\lambda u^{\lambda-1}t^2}{\pi} \int_0^{\pi/2} \frac{(\sin \theta)^{2\lambda-1}}{[(|y|-u)^2+t^2+2|y|u(1-\cos \theta)]^{\lambda+2}} d\theta \\
                & - \frac{4\lambda(\lambda+1) (|y|u)^\lambda }{\pi} \int_0^{\pi/2} \frac{(\sin \theta)^{2\lambda-1}[(u-|y|) +|y|(1-\cos \theta)]}{[(|y|-u)^2+t^2+2|y|u(1-\cos \theta)]^{\lambda+2}} d\theta\\
                & + \frac{8\lambda(\lambda+1)(\lambda+2) (|y|u)^\lambda t^2}{\pi} \int_0^{\pi/2} \frac{(\sin \theta)^{2\lambda-1}[(u-|y|) +|y|(1-\cos \theta)]}{[(|y|-u)^2+t^2+2|y|u(1-\cos \theta)]^{\lambda+3}} d\theta.
        \end{align*}
        Since $\sin \theta \sim \theta$ and $2(1-\cos \theta) \sim \theta^2$, when $\theta \in [0,\pi/2]$, it follows that
        \begin{align*}
            |\partial_t \partial_u P_{t,1}^\lambda(y,u)|
                \leq  &C \left(u^{\lambda-1} \int_0^{\pi/2} \frac{|y|^\lambda \theta^{2\lambda-1}}{[(|y|-u)^2+t^2+|y|u\theta^2]^{\lambda+1}} d\theta
                         + u^{\lambda} \int_0^{\pi/2} \frac{|y|^\lambda \theta^{2\lambda-1}}{[(|y|-u)^2+t^2+|y|u\theta^2]^{\lambda+3/2}} d\theta \right)\\
                = & C \left( A_t^\lambda(y,u) + B_t^\lambda(y,u) \right), \quad u,t \in (0,\infty), \ y \in \mathbb{R}.
        \end{align*}

        We analyze firstly $A_t^\lambda(y,u)$. Assume $0 < \lambda \leq 1$. We get
        \begin{align*}
            \frac{|y|^\lambda \theta^{\lambda}}{[(|y|-u)^2+t^2+|y|u\theta^2]^{\lambda+1}}
                \leq C \frac{(|y|u\theta^2)^{\lambda/2} }{[(|y|-u)^2+t^2+|y|u\theta^2]^{\lambda+1}}
                \leq \frac{C}{(||y|-u|+t)^{\lambda+2}}, \quad 0 < |y| < 2u,
        \end{align*}
        and
        \begin{align*}
            \frac{|y|^\lambda }{[(|y|-u)^2+t^2+|y|u\theta^2]^{\lambda+1}}
                \leq C \frac{|y|^{\lambda} }{[(|y|-u)^2+t^2]^{\lambda+1}}
                \leq \frac{C}{(||y|-u|+t)^{\lambda+2}}, \quad 0 < 2u < |y|.
        \end{align*}
        Hence,
        \begin{align*}
            A_t^\lambda(y,u)
                \leq C \frac{u^{\lambda-1}}{(||y|-u|+t)^{\lambda+2}}, \quad u,t \in (0,\infty), \ y \in \mathbb{R}.
        \end{align*}
        By proceeding as in \eqref{bound1} we obtain
        $$\|t A_t^\lambda(y,u) \|_{L^q\left(\G(x),\frac{dtdy}{t^2}\right)}
            \leq C \|t A_t^\lambda(y,u) \|_{L^q\left(\G_+(x),\frac{dtdy}{t^2}\right)}
            \leq C \frac{u^{\lambda-1}}{|x-u|^{\lambda+1}}, \quad u,x \in (0,\infty).$$
        Also notice that
        \begin{align*}
            \int_{(0,\infty) \setminus 2I} \frac{dx}{|x-z_0|^{\lambda+1}}
                \leq \frac{C}{|I|^\lambda},
        \end{align*}
        and that $|x-u| \geq |x-z_0|/2$ provided that $u \in I$ and $x \in (0,\infty)\setminus 2I$. By combining these facts we conclude
        \begin{align} \label{5.7}
            \frac{1}{|I|} \int_{(0,\infty) \setminus 2I} &\int_I \left| \int_{z_0}^z \|t A_t^\lambda(y,u) \|_{L^q\left(\G(x),\frac{dtdy}{t^2}\right)} du \right| dz dx
                \leq   \frac{C}{|I|} \int_{(0,\infty) \setminus 2I} \frac{dx}{|x-z_0|^{\lambda+1}}  \int_I \left| \int_{z_0}^z  u^{\lambda-1} du \right| dz \nonumber \\
                \leq &  \frac{C}{|I|^{\lambda+1}}\int_I \left| z^\lambda - z_0^\lambda \right| dz
                =  \frac{C}{|I|^{\lambda+1}} \left(\int_0^{|I|/2}  [(z_0+w)^\lambda - z_0^\lambda] dw +  \int_{-|I|/2}^0  [z_0^\lambda  - (z_0+w)^\lambda] dw \right) \nonumber\\
                \leq &   \frac{C}{|I|^{\lambda+1}} \int_0^{|I|/2}  w^\lambda dw
                \leq C.
        \end{align}
        We have used that $(a+b)^\alpha \leq a^\alpha + b^\alpha$, when $a,b>0$ and $0 < \alpha \leq 1$.\\

        Suppose now $\lambda>1$. We have that
        \begin{align*}
            A_t^\lambda(y,u)
                = &  \int_0^{\pi/2} \left(\frac{|y|u\theta^2}{(|y|-u)^2+t^2+|y|u\theta^2} \right)^{\lambda-1}
                         \frac{|y|\theta}{(|y|-u)^2+t^2+|y|u\theta^2)^2} d\theta \\
                \leq C & \int_0^{\pi/2}  \frac{|y|\theta}{[(|y|-u)^2+t^2+|y|u\theta^2]^{2}}d\theta
                =C A_t^1(y,u) , \quad u,t \in (0,\infty), \ y \in \mathbb{R}.
        \end{align*}
        Then, by using what we have proved in the above case we get
        \begin{align}\label{5.8}
            \frac{1}{|I|} \int_{(0,\infty) \setminus 2I} &\int_I \left| \int_{z_0}^z \|t A_t^\lambda(y,u) \|_{L^q\left(\G(x),\frac{dtdy}{t^2}\right)} du \right| dz dx
                \leq C.
        \end{align}

        Finally, to treat the term $B_t^\lambda(y,u)$ we make the change of variables $\theta = \phi \sqrt{(|y|-u)^2 + t^2}/\sqrt{|y|u}$ and we obtain
        \begin{align*}
            B_t^\lambda(y,u)
                \leq &  \frac{C}{(||y|-u|+t)^3} \int_0^\infty \frac{\phi^{2\lambda-1}}{(1+\phi^2)^{\lambda+1/2}} d\phi
                \leq  \frac{C}{(||y|-u|+t)^3}, \quad u,t \in (0,\infty), \ y \in \mathbb{R}.
        \end{align*}
        As above it follows that
        \begin{align}\label{5.9}
            \frac{1}{|I|} \int_{(0,\infty) \setminus 2I} &\int_I \left| \int_{z_0}^z \|t B_t^\lambda(y,u) \|_{L^q\left(\G(x),\frac{dtdy}{t^2}\right)} du \right| dz dx
                \leq C.
        \end{align}
        Note that the constants $C$ in \eqref{5.7}, \eqref{5.8} and \eqref{5.9} do not depend on $I$. Thus \eqref{5.6} is shown.\\

        Putting together all the estimations that we have just obtained, \eqref{L1-atom} is proved and the proof of this proposition is finished.
    \end{proof}

    \subsection{Appendix 2}\label{subsec:app2}

        In this part we study in detail the operator $\Phi_N$, $N \in \mathbb{N}$, which appears in Subsections~\ref{subsec:1->2} and \ref{subsec2:1->2}.
        We prove that the sequence $\{\Phi_N\}_{N \in \mathbb{N}}$ can be seen as a uniform (in a suitable sense) family of vector
        valued Calderón-Zygmund operators. Consequently, the mapping properties that we need for $\Phi_N$, $N \in \mathbb{N}$, follow from the general theory
        (\cite{RbRuT}).\\

        Let $\B$ be a Banach space, $1 < q < \infty$ and $N \in \mathbb{N}$. For every $h \in L^\infty_c\left((0,\infty),L^q\left(\G_N(0), \dfrac{dydt}{t^2},\B\right)\right)$
        we define
        $$ \Phi_N(h)(x,y,t)
            = \int_0^\infty  \int_{\G_N(0)} k^\lambda_{s,t}(x,y;u,v)  h(v,u,s)\frac{dsdu}{s^2} dv,
            \quad x,t \in (0,\infty), \ y \in \mathbb{R},$$
        where
        $$ k^\lambda_{s,t}(x,y;u,v) = \int_0^\infty t \partial_t P_t^\lambda(x+y,z) s \partial_s P_s^\lambda(v+u,z) dz,
        \quad v,x,s,t \in (0,\infty), \ u,y \in \mathbb{R}.$$
        In Subsection~\ref{subsec:1->2} it was proved that the integral defining $ \Phi_N(h)(x,y,t)$ is absolutely convergent for every
        $h \in L^\infty_c\left((0,\infty),L^q\left(\G_N(0), \dfrac{dydt}{t^2},\B\right)\right)$, $x,t \in (0,\infty)$, $y \in \mathbb{R}$
        and $2 \leq q < \infty$. Notice that this property is also true for $1<q<2$.\\

        To simplify the notation we write in the sequel $\F^q=L^q\left(\G(0),\dfrac{dt dy}{y^2},\B\right)$.

        \begin{Lem}\label{Lem5.1}
            Let $\B$ be a Banach space, $\lambda>0$ and $1<q<\infty$. Then, for every $N \in \mathbb{N}$, the operator $\Phi_N$ is bounded from
            $L^q((0,\infty),\F^q)$ into itself. Moreover, there exists $C>0$ such that, for every $N \in \mathbb{N}$,
            $$\|\Phi_N(g)\|_{L^q((0,\infty),\F^q)}
                \leq C \|g\|_{L^q((0,\infty),\F^q)}, \quad g \in L^q((0,\infty),\F^q).$$
        \end{Lem}

        \begin{proof}
            Let $N \in \mathbb{N}$ and $g \in L^q((0,\infty),\F^q)$. Hölder's inequality implies that
            \begin{align*}
                |\Phi_N(g)(x,y,t)|
                    \leq & \left(\int_0^\infty  \int_{\G_N(0)} |k^\lambda_{s,t}(x,y;u,v)| \frac{dsdu}{s^2} dv \right)^{1/q'} \\
                         & \times  \left(\int_0^\infty  \int_{\G_N(0)} |k^\lambda_{s,t}(x,y;u,v)| \ \|g(v,u,s)\|_\B^q \frac{dsdu}{s^2} dv \right)^{1/q},
                         \quad x,t \in (0,\infty), \ y \in \mathbb{R}.
            \end{align*}
           By \eqref{kst} it follows that
            \begin{align}\label{int_kst}
                \int_0^\infty  \int_{\G_N(0)} |k^\lambda_{s,t}(x,y;u,v)| \frac{dsdu}{s^2} dv
                    \leq & C \int_{\G(0)}\int_0^\infty   \frac{st}{(z+s+t)^3} dz \frac{dsdu}{s^2}
                    \leq  C \int_0^\infty  \frac{t}{(s+t)^2}ds
                    \leq C. %, \quad x,t \in (0,\infty), \ y \in \mathbb{R}.
            \end{align}
            Note that $C$ does not depend on $N$.
            Now \eqref{int_kst} leads to
            \begin{align*} %\label{boundPhi}
                \|\Phi_N(g)\|_{L^q((0,\infty),\F^q)}^q
                    \leq C & \int_{\G_N(0)}  \int_0^\infty \left( \int_0^\infty \int_{\G(0)} |k^\lambda_{s,t}(x,y;u,v)|   \frac{dtdy}{t^2} dx \right) \|g(v,u,s)\|_\B^q \frac{dsdu}{s^2}  dv \\
                    \leq C & \|g\|^q_{L^q((0,\infty),\F^q)}.
            \end{align*}
        \end{proof}

        In the next lemma we introduce, for every $N \in \mathbb{N}$, a family $\{K_N^\lambda(x,v)\}_{x,v \in (0,\infty), \ x \neq v}$ of bounded
        operators in $\F^q$. We prove that $K_N^\lambda(x,v)$, $x,v \in (0,\infty)$, $x \neq v$, satisfies the standard Calderón-Zygmund conditions
        uniformly in $N \in \mathbb{N}$.

        \begin{Lem}\label{Lem5.2}
            Let $\B$ be a Banach space, $\lambda>1$ and $1<q<\infty$. For every $N \in \mathbb{N}$, and $x,v \in (0,\infty)$, $x \neq v$, we define
            $$K_N^\lambda(x,v)(h)(y,t)
                = \int_{\G_N(0)} k^\lambda_{s,t}(x,y;u,v) h(u,s)\frac{dsdu}{s^2}, \quad h \in \F^q.$$
            Then,
            \begin{itemize}
                \item[$(a)$] $K_N^\lambda(x,v)$, $N \in \mathbb{N}$, is bounded from $\F^q$ into itself and there exists $C>0$ such that for every $N \in \mathbb{N}$
                and $x,v \in (0,\infty)$, $x \neq v$,
                $$ \|K_N^\lambda(x,v)(h)\|_{\F^q} \leq  \frac{C}{|x-v|} \|h\|_{\F^q}, \quad h \in \F^q.$$
                \item[$(b)$] There exits $C>0$ such that, for every $N \in \mathbb{N}$,
                $$ \|(K_N^\lambda(x_1,v) - K_N^\lambda(x_2,v))(h)\|_{\F^q} \leq C \frac{|x_1-x_2|}{|x_1-v|^2}\|h\|_{\F^q},$$
                being $h \in \F^q$ and $|x_1-v|>2|x_1-x_2|$, $x_1,x_2,v \in (0,\infty)$.
                \item[$(c)$] There exits $C>0$ such that, for every $N \in \mathbb{N}$,
                $$ \|(K_N^\lambda(x,v_1) - K_N^\lambda(x,v_2))(h)\|_{\F^q} \leq C \frac{|v_1-v_2|}{|v_1-x|^2}\|h\|_{\F^q},$$
                being $h \in \F^q$ and $|x-v_1|>2|v_1-v_2|$, $x,v_1,v_2 \in (0,\infty)$.
            \end{itemize}
        \end{Lem}

        \begin{proof}
            $(a)$ Note firstly that if $h \in \F^q$ we have that, for every $N \in \mathbb{N}$ and $v,x \in (0,\infty)$, with $x \neq v$,
            \begin{align*}
                \|K_N^\lambda(x,v)(h)\|_{\F^q}
                    \leq & \left( \int_{\G(0)}
                        \int_{\G_N(0)} \|h(u,s)\|_\B^q \frac{dsdu}{s^2}
                        \left(\int_{\G_N(0)} |k^\lambda_{s,t}(x,y;u,v)|^{q'} \frac{dsdu}{s^2}  \right)^{q/q'}
                        \frac{dtdy}{t^2} \right)^{1/q} \\
                    \leq  & C \left( \int_{\G_+(x)}
                        \left(\int_{\G_+(v)} |k^\lambda_{s,t}(0,y;u,0)|^{q'} \frac{dsdu}{s^2}  \right)^{q/q'}
                        \frac{dtdy}{t^2} \right)^{1/q} \|h\|_{\F^q},
            \end{align*}
            being $C>0$ independent of $N$. Then, $(a)$ is established when we prove that for a certain $C>0$
            \begin{equation}\label{5.11}
                \left( \int_{\G_+(x)}\left(\int_{\G_+(v)} |k^\lambda_{s,t}(0,y;u,0)|^{q'} \frac{dsdu}{s^2}  \right)^{q/q'} \frac{dtdy}{t^2} \right)^{1/q}
                    \leq \frac{C}{|x-v|}, \quad x,v \in (0,\infty), \ x \neq v.
            \end{equation}

            We write
            $$k^\lambda_{s,t}(0,y;u,0)
                = ts \partial_r^2 P_{r,1}^\lambda(y,u)_{|_{r=t+s}} + ts \partial_r^2 P_{r,2}^\lambda(y,u)_{|_{r=t+s}},
                \quad y,u,t,s \in (0,\infty),$$
            where $P_{r,1}^\lambda(y,u)$ and $P_{r,2}^\lambda(y,u)$ are given by \eqref{7.2}.
            We have that
            \begin{align*}
                \partial_r^2 P_{r,1}^\lambda(y,u)
                    = & -\frac{12\lambda(\lambda+1) (yu)^\lambda r}{\pi} \int_0^{\pi/2} \frac{(\sin \theta)^{2\lambda-1}}{[(y-u)^2+r^2+2yu(1-\cos \theta)]^{\lambda+2}} d\theta  \\
                      & + \frac{8\lambda(\lambda+1)(\lambda+2) (yu)^\lambda r^3}{\pi} \int_0^{\pi/2} \frac{(\sin \theta)^{2\lambda-1}}{[(y-u)^2+r^2+2yu(1-\cos \theta)]^{\lambda+3}} d\theta \\
                    = & \mathcal{L}_{r,1}^{\lambda}(y,u)  +  \mathcal{L}_{r,2}^{\lambda}(y,u), \quad r,y,u \in (0,\infty).
            \end{align*}
            It is clear that $ |\mathcal{L}_{r,2}^{\lambda}(y,u)| \leq C |\mathcal{L}_{r,1}^{\lambda}(y,u)|$, $r,y,u \in (0,\infty)$.
            Moreover, by taking into account that $\sin \theta \sim \theta$ and $2(1-\cos \theta) \sim \theta^2$, when $\theta \in [0,\pi/2]$,
            and  making the change of variables $\theta = \sqrt{|y-u|^2+r^2}\phi/\sqrt{uy}$, we get
            $$ |\mathcal{L}_{r,1}^{\lambda}(y,u)|
                \leq C r (yu)^\lambda \int_0^{\pi/2} \frac{\theta^{2\lambda-1}}{((y-u)^2+r^2+yu\theta^2)^{\lambda+2}} d\theta
                \leq  \frac{C}{(|y-u|+r)^3}, \quad r,y,u \in (0,\infty).$$
            Then, since $|y-u| + t+ s \sim |x-v|+t+s$, when $(y,t) \in \G_+(x)$ and $(u,s) \in \G_+(v)$, we obtain
            \begin{align*}
                & \left( \int_{\G_+(x)}\left(\int_{\G_+(v)} |st \partial_r^2 P_{r,1}^\lambda(y,u)_{|r=s+t}|^{q'} \frac{dsdu}{s^2}  \right)^{q/q'}\frac{dtdy}{t^2} \right)^{1/q} \\
                & \qquad \qquad \leq C  \left( \int_{\G_+(x)}t^{q-2}\left(\int_{\G_+(v)} \frac{s^{q'-2}}{(|x-v|+s+t)^{3q'}} dsdu  \right)^{q/q'}dtdy \right)^{1/q} \\
                & \qquad \qquad \leq C  \left( \int_{\G_+(x)}t^{q-2}\left(\int_0^\infty \frac{s^{q'-1}}{(|x-v|+s+t)^{3q'}} ds  \right)^{q/q'}dtdy \right)^{1/q} \\
                & \qquad \qquad \leq C  \left( \int_{\G_+(x)}t^{q-2}\left(\int_0^\infty \frac{ds}{(|x-v|+s+t)^{2q'+1}} ds  \right)^{q/q'}dtdy \right)^{1/q} \\
                & \qquad \qquad \leq C  \left( \int_0^\infty \int_{|y-x|\leq t}\frac{t^{q-2}}{(|x-v|+t)^{2q}}  dtdy \right)^{1/q}
                \leq  \frac{C}{|x-v|}, \quad x,v \in (0,\infty), \ x \neq v.
            \end{align*}
            In a similar way we can get that
            $$\left( \int_{\G_+(x)}\left(\int_{\G_+(v)} |st \partial_r^2 P_{r,2}^\lambda(y,u)_{|r=s+t}|^{q'} \frac{dsdu}{s^2}  \right)^{q/q'}\frac{dtdy}{t^2} \right)^{1/q}
                \leq \frac{C}{|x-v|}, \quad x,v \in (0,\infty), \ x \neq v.$$
            Hence \eqref{5.11} is established and $(a)$ is proved.\\

            $(b)$ By proceeding as above and using Minkowski's inequality we can see that
            \begin{align*}
                \|K_N^\lambda(x_1,v)(h) & - K_N^\lambda(x_2,v)(h)\|_{\F^q} \\
                    \leq & \left( \int_{\G(0)} \left(\int_{\G(0)} |k^\lambda_{s,t}(x_1,y;u,v) - k^\lambda_{s,t}(x_2,y;u,v)|^{q'} \frac{dsdu}{s^2}  \right)^{q/q'}
                            \frac{dtdy}{t^2} \right)^{1/q} \|h\|_{\F^q}\\
                    \leq & \left( \int_{\G(0)} \left(\int_{\G(0)} \left| \int_{x_1}^{x_2} \left| \partial_z k^\lambda_{s,t}(z,y;u,v) \right| dz \right|^{q'} \frac{dsdu}{s^2}  \right)^{q/q'}
                            \frac{dtdy}{t^2} \right)^{1/q} \|h\|_{\F^q} \\
                    \leq & \left| \int_{x_1}^{x_2} \left( \int_{\G(0)} \left(\int_{\G(0)}  \left| \partial_z k^\lambda_{s,t}(z,y;u,v) \right|^{q'} \frac{dsdu}{s^2}  \right)^{q/q'}
                            \frac{dtdy}{t^2} \right)^{1/q}dz \right| \|h\|_{\F^q}.
            \end{align*}
            Hence, $(b)$ is shown when we prove that
            \begin{equation}\label{5.12}
                \left| \int_{x_1}^{x_2} \left( \int_{\G(0)} \left(\int_{\G(0)}  \left| \partial_z k^\lambda_{s,t}(z,y;u,v) \right|^{q'} \frac{dsdu}{s^2}  \right)^{q/q'}
                            \frac{dtdy}{t^2} \right)^{1/q}dz \right|
                    \leq C \frac{|x_1-x_2|}{|x_1-v|^2},
            \end{equation}
            for every $x_1,x_2,v \in (0,\infty)$ such that $|x_1-v|>2|x_1-x_2|$.\\

            From now on, we take into account that $\lambda>1$.  Suppose that $x_1,x_2,v \in (0,\infty)$ such that $|x_1-v|>2|x_1-x_2|$. We can write
            \begin{align}\label{5.13}
                & \left| \int_{x_1}^{x_2} \left( \int_{\G(0)} \left(\int_{\G(0)}  \left| \partial_z k^\lambda_{s,t}(z,y;u,v) \right|^{q'} \frac{dsdu}{s^2}  \right)^{q/q'}
                            \frac{dtdy}{t^2} \right)^{1/q}dz \right| \nonumber \\
                & \qquad  \leq C \left| \int_{x_1}^{x_2} \left( \int_{\G_+(z)} \left(\int_{\G_+(v)}  \left| \partial_y k^\lambda_{s,t}(0,y;u,0) \right|^{q'} \frac{dsdu}{s^2}  \right)^{q/q'}
                            \frac{dtdy}{t^2} \right)^{1/q}dz \right|.
            \end{align}
            By keeping the notation in the proof of $(a)$, straightforward manipulations lead to
            \begin{align*}
                \partial_y \partial_r^2 P_{r,1}^\lambda(y,u)
                    = & -\frac{12\lambda^2(\lambda+1) (yu)^{\lambda-1}u r}{\pi} \int_0^{\pi/2} \frac{(\sin \theta)^{2\lambda-1}}{[(y-u)^2+r^2+2yu(1-\cos \theta)]^{\lambda+2}} d\theta  \\
                      & +\frac{24\lambda(\lambda+1)(\lambda+2) (yu)^\lambda r}{\pi} \int_0^{\pi/2} \frac{(\sin \theta)^{2\lambda-1}((y-u)+u(1-\cos \theta))}{[(y-u)^2+r^2+2yu(1-\cos \theta)]^{\lambda+3}} d\theta  \\
                      & +\frac{8\lambda^2(\lambda+1)(\lambda+2) (yu)^{\lambda-1}u r^3}{\pi} \int_0^{\pi/2} \frac{(\sin \theta)^{2\lambda-1}}{[(y-u)^2+r^2+2yu(1-\cos \theta)]^{\lambda+3}} d\theta\\
                      & - \frac{16\lambda(\lambda+1)(\lambda+2)(\lambda+3) (yu)^\lambda r^3}{\pi} \int_0^{\pi/2} \frac{(\sin \theta)^{2\lambda-1}((y-u)+u(1-\cos \theta))}{[(y-u)^2+r^2+2yu(1-\cos \theta)]^{\lambda+4}} d\theta \\
                    = & \mathcal{L}_{r,1,1}^{\lambda}(y,u)+ \mathcal{L}_{r,1,2}^{\lambda}(y,u)
                        + \mathcal{L}_{r,2,1}^{\lambda}(y,u) +  \mathcal{L}_{r,2,2}^{\lambda}(y,u) , \quad r,y,u \in (0,\infty).
            \end{align*}
            If we define
            $$\mathcal{L}_{r,3}^{\lambda}(y,u)
                = (yu)^\lambda \int_0^{\pi/2} \frac{(\sin \theta)^{2\lambda-1}}{[(y-u)^2+r^2+2yu(1-\cos \theta)]^{\lambda+2}} d\theta, \quad r,y,u \in (0,\infty),$$
            we have the following relations,
            \begin{itemize}
                \item $\mathcal{L}_{r,1,2}^{\lambda}(y,u) \leq C  \left(\mathcal{L}_{r,1,1}^{\lambda}(y,u)  + \mathcal{L}_{r,3}^{\lambda}(y,u)\right), \quad r,y,u \in (0,\infty),$
                \item $\mathcal{L}_{r,2,1}^{\lambda}(y,u) \leq C  \mathcal{L}_{r,1,1}^{\lambda}(y,u) , \quad r,y,u \in (0,\infty),$
                \item $\mathcal{L}_{r,2,2}^{\lambda}(y,u)  \leq C  \left(\mathcal{L}_{r,1,1}^{\lambda}(y,u) + \mathcal{L}_{r,3}^{\lambda}(y,u)\right), \quad r,y,u \in (0,\infty).$
            \end{itemize}
            Therefore, it is sufficient to analyze $\mathcal{L}_{r,1,1}^{\lambda}(y,u)$ and $\mathcal{L}_{r,3}^{\lambda}(y,u)$, $r,y,u \in (0,\infty)$.\\

            Now we can see
            $$\mathcal{L}_{s+t,3}^{\lambda}(y,u)
                \leq \dfrac{C}{(|y-u|+s+t)^4}
                \leq \dfrac{C}{(|z-v|+s+t)^4}, \quad (y,t) \in \G_+(z), \ (u,s) \in \G_+(v).$$
            Moreover, we have that
            $$\mathcal{L}_{s+t,1,1}^{\lambda}(y,u)
                \leq C \dfrac{1}{|z-v|(|z-v|+s+t)^3}, \quad y \geq |z-v|, \  (y,t) \in \G_+(z), \ (u,s) \in \G_+(v),$$
            and, since $\lambda>1$,
            \begin{align*}
                \mathcal{L}_{s+t,1,1}^{\lambda}(y,u)
                    \leq & C (yu)^{\lambda-1}u (s+t) \int_0^{\pi/2} \frac{\theta^{2(\lambda-1)}}{[(y-u)^2+(s+t)^2+yu\theta^2]^{\lambda+2}} d\theta \\
                    \leq & C \frac{u(s+t)}{(|y-u|+s+t)^6}
                    \leq  C \frac{|u-y|+y}{(|y-u|+s+t)^5}
                    \leq \frac{C}{(|z-v|+s+t)^4} \\
                    \leq  & \frac{C}{|z-v|(|z-v|+s+t)^3}, \quad y \leq |z-v|, \  (y,t) \in \G_+(z), \ (u,s) \in \G_+(v).
            \end{align*}
            The same computations made in the proof of  $(a)$ give us
            $$\left( \int_{\G_+(z)}
                            \left(\int_{\G_+(v)} |st \partial_y \partial_r^2 P_{r,1}^\lambda(y,u)_{|r=s+t} |^{q'} \frac{dsdu}{s^2}  \right)^{q/q'}
                        \frac{dtdy}{t^2} \right)^{1/q}
                        \leq \frac{C}{|z-v|^2}, \quad v,z \in (0,\infty).$$
            Similarly we can obtain
            $$\left( \int_{\G_+(z)}
                            \left(\int_{\G_+(v)} |st \partial_y \partial_r^2 P_{r,2}^\lambda(y,u)_{|r=s+t} |^{q'} \frac{dsdu}{s^2}  \right)^{q/q'}
                        \frac{dtdy}{t^2} \right)^{1/q}
                        \leq \frac{C}{|z-v|^2}, \quad v,z \in (0,\infty).$$
            Hence, we conclude that
            $$\left( \int_{\G_+(z)}
                            \left(\int_{\G_+(v)} |\partial_y k^\lambda_{s,t}(0,y;u,0) |^{q'} \frac{dsdu}{s^2}  \right)^{q/q'}
                        \frac{dtdy}{t^2} \right)^{1/q}
                        \leq \frac{C}{|z-v|^2}, \quad v,z \in (0,\infty).$$
            From \eqref{5.13} it follows that
            \begin{align*}
              & \left| \int_{x_1}^{x_2} \left( \int_{\G(0)} \left(\int_{\G(0)}  \left| \partial_z k^\lambda_{s,t}(z,y;u,v) \right|^{q'} \frac{dsdu}{s^2}  \right)^{q/q'}
                            \frac{dtdy}{t^2} \right)^{1/q}dz \right|
               \leq C \left| \int_{x_1}^{x_2} \frac{1}{|z-v|^2} dz \right|
                        \leq C \frac{|x_1-x_2|}{|v-x_1|^2},
            \end{align*}
            for each $x_1,x_2,v \in (0,\infty)$ such that $|x_1-v|>2|x_1-x_2|.$\\

            Thus, \eqref{5.12} is shown and the proof of $(b)$ is completed.\\

            $(c)$ The proof of $(c)$ is essentially the same one of $(b)$.
        \end{proof}

        We now obtain a representation of the operator $\Phi_N$ as a vector valued integral operator, for every $N \in \mathbb{N}$.

        \begin{Lem}\label{Lem5.3}
            Let $\B$ be a Banach space, $\lambda>0$, $1<q<\infty$ and $N \in \mathbb{N}$. We denote by $K_N^\lambda(x,v)$, $x,v \in (0,\infty)$, $x \neq v$,
            the operator introduced in Lemma~\ref{Lem5.2}. Then,
            \begin{equation}\label{5.14}
                \Phi_N(g)(x)
                    = \int_0^\infty K_N^\lambda(x,v)(g(v)) dv, \quad \text{a.e. } x \notin \supp(g), \
                        g \in L^\infty_{c}(0,\infty) \otimes \left( L^q(\G(0),\frac{dtdy}{t^2}) \otimes \B \right),
            \end{equation}
            where if $g \in L^\infty_{c}((0,\infty),\F^q)$ we represent
            \begin{itemize}
                \item for every $v \in (0,\infty)$, $g(v)(y,t)=g(v,y,t)$, $(y,t) \in \G(0)$,
                \item for every $x \in (0,\infty)$, $\Phi_N(g)(x)(u,s)=\Phi_N(g)(x,u,s)$, $(u,s) \in \G(0)$.
            \end{itemize}
            The integral in \eqref{5.14} is understood in the $\F^q$-Bochner sense.
        \end{Lem}

        \begin{proof}
            It is sufficient to show the result when $\B$ has finite dimension.\\
            Let $g \in L^\infty_{c}((0,\infty),\F^q)$. We are going to see that, for almost all $x \notin \supp(g)$,
            \begin{equation}\label{5.15}
                \int_0^\infty  \int_{\G_N(0)} k^\lambda_{s,t}(x,y;u,v)  g(v,u,s)\frac{dsdu}{s^2} dv
                    = \left(\int_0^\infty K_N^\lambda(x,v)(g(v)) dv\right)(x,y,t),
            \end{equation}
            in the sense of equality in $\F^q$. Note that the $\F^q$-Bochner integral in the right hand side is absolutely convergent
            for every $x \notin \supp(g)$. Indeed, according to Lemma~\ref{Lem5.2}, $(a)$, we get
            \begin{align*}
                \int_0^\infty \|K_N^\lambda(x,v)(g(v))\|_{\F^q} dv
                    \leq &C \int_{\supp(g)} \frac{\|g(v)\|_{\F^q}}{|x-v|}dv \\
                    \leq &C \|g\|_{L^\infty((0,\infty),\F^q)} \int_{\supp(g)} \frac{dv}{|x-v|}
                    < \infty, \quad x \notin \supp(g).
            \end{align*}

            In order to show \eqref{5.15} it is enough to see that, for every $H \in \left( L^q((0,\infty),\F^q)\right)^*$ and $x \notin \supp(g)$
            \begin{align}\label{dualidad}
                \langle H(x,y,t), \int_0^\infty  \int_{\G_N(0)} k^\lambda_{s,t}(x,y;u,v)  g(v,u,s)\frac{dsdu}{s^2} dv\rangle
                    = \langle H, \int_0^\infty K_N^\lambda(x,v)(g(v)) dv \rangle.
            \end{align}

            Let $H \in \left( L^q((0,\infty),\F^q)\right)^*$ and $x \notin \supp(g)$. By \cite[Corollary III.2.13]{DU},
            $\left( L^q((0,\infty),\F^q)\right)^*=L^{q'}((0,\infty),(\F^q)^*)$, where $(\F^q)^*=L^{q'}\left(\G(0),\frac{dt dy}{t^2}, \B^*\right)$.
            Hence, there exists $h \in L^{q'}((0,\infty),(\F^q)^*)$ such that
            $$\langle H, G \rangle
                = \int_0^\infty \langle h(x),G(x) \rangle_{(\F^{q})^* \times \F^q } dx, \quad G \in L^q((0,\infty),\F^q).$$
            Hence, we can write
            \begin{align*}
                \langle H(x,y,t), \int_0^\infty  \int_{\G_N(0)} & k^\lambda_{s,t}(x,y;u,v)  g(v,u,s)\frac{dsdu}{s^2} dv\rangle \\
                    = & \int_0^\infty  \int_{\G(0)} \int_0^\infty  \int_{\G_N(0)} k^\lambda_{s,t}(x,y;u,v) \langle h(x,y,t), g(v,u,s) \rangle \frac{dsdu}{s^2}  dv \frac{dtdy}{t^2} dx.
            \end{align*}
            Moreover, well-known properties of the Bochner integrals lead us to
            \begin{align*}
                \langle H, \int_0^\infty K_N^\lambda(x&,v)(g(v)) dv \rangle
                    = \int_0^\infty \langle H , K_N^\lambda(x,v)(g(v)) \rangle dv \\
                    = & \int_0^\infty  \int_0^\infty \int_{\G(0)}   \int_{\G_N(0)} k^\lambda_{s,t}(x,y;u,v) \langle h(x,y,t), g(v,u,s) \rangle \frac{dsdu}{s^2} \frac{dtdy}{t^2}dx dv.
            \end{align*}
            To obtain \eqref{dualidad} we only need to show that the last integral is absolutely convergent.
            For this purpose we apply Hölder's inequality and \eqref{int_kst} as follows
            \begin{align*}
                \int_0^\infty & \int_0^\infty   \int_{\G(0)}   \int_{\G_N(0)} |k^\lambda_{s,t}(x,y;u,v)| \ \|g(v,u,s)\|_\B \frac{dsdu}{s^2} \|h(x,y,t)\|_{\B^*}  \frac{dtdy}{t^2}dx dv \\
                    \leq & \left(\int_0^\infty  \int_{\G(0)} \int_0^\infty  \int_{\G_N(0)}  |k^\lambda_{s,t}(x,y;u,v)| \|h(x,y,t)\|_{\B^*}^{q'}  \frac{dsdu}{s^2}dv \frac{dtdy}{t^2}dx  \right)^{1/q'}\\
                    & \times \left( \int_0^\infty  \int_{\G(0)} \int_0^\infty  \int_{\G_N(0)} |k^\lambda_{s,t}(x,y;u,v)| \|g(v,u,s)\|_{\B}^q \frac{dsdu}{s^2}dv  \frac{dtdy}{t^2}dx \right)^{1/q} \\
                    \leq & C  \|h\|_{L^{q'}((0,\infty),(\F^q)^*)} \|g\|_{L^q((0,\infty),\F^q)}
                    < \infty.
            \end{align*}
            The proof is finished.
        \end{proof}

        By using Lemmas~\ref{Lem5.1}, \ref{Lem5.2} and \ref{Lem5.3} and as a consequence of the theory of vector valued Calderón-Zygmund
        operators (see \cite{RbRuT})
        we obtain the following result that we used in the proof of Theorems~\ref{Th_3.1_OX} and \ref{Th_4.1_OX}.

        \begin{Prop}\label{Pro5.2}
             Let $\B$ be a Banach space, $\lambda>1$, $1<q<\infty$. Then, for each $N \in \mathbb{N}$, the operator $\Phi_N$ can be extended
             \begin{itemize}
                \item[$(a)$] to $L^p((0,\infty), \F^q)$ as a bounded operator from $L^p((0,\infty), \F^q)$ into itself, for every $1<p<\infty$;
                \item[$(b)$] to $H^1((0,\infty), \F^q)$ as a bounded operator from $H^1((0,\infty), \F^q)$ into $L^1((0,\infty), \F^q)$.
             \end{itemize}
             Moreover, for every $1<p<\infty$ there exists $C_p>0$ such that
             $$\|\Phi_N(g)\|_{L^p((0,\infty), \F^q)}
                \leq C_p \|g\|_{L^p((0,\infty), \F^q)}, \quad g \in L^p((0,\infty), \F^q),$$
             and there exists $C_1>0$ such that
             $$\|\Phi_N(g)\|_{L^1((0,\infty), \F^q)}
                \leq C_1 \|g\|_{H^1((0,\infty), \F^q)}, \quad g \in H^1((0,\infty), \F^q),$$
             for every $N \in \mathbb{N}$.
        \end{Prop}

%\newpage
    %%%%%%%%%%%%%%%%%%%%%%%%%%%%%%%%%%%%%%%%%%%%%%%%%%%%%%%%%%%%%%%%%%%%%%%%%%%%%%%%%%%%%%%%%%%%%%%%%%%%%%%%%%%%%%%%%%%%

\end{document}